\documentclass[a4paper,11pt]{article}
\usepackage[utf8]{inputenc}
\usepackage{amsmath}
\usepackage{amsfonts}
\usepackage{amsthm}
\usepackage{amssymb}
\usepackage{color}
\usepackage{enumerate}
\usepackage{graphicx}
\usepackage{caption}
\usepackage{subcaption}
\usepackage{mathtools}
\usepackage[all]{xy}

\usepackage{marginnote}
\usepackage{hyperref}
\usepackage[UKenglish]{babel}
\usepackage{titlefoot}
\usepackage{csquotes}
\usepackage[
	hyperref=auto,
	backend=biber,
	style=alphabetic,
	giveninits=true,
	minalphanames=4,
	maxalphanames=4,
	maxbibnames=999,
	doi=false,
	arxiv=false
]{biblatex}
\usepackage[capitalize]{cleveref}

\DeclareFieldFormat[article]{title}{\emph{#1},}
\DeclareFieldFormat[thesis]{title}{\emph{#1},}
\DeclareFieldFormat[article]{journaltitle}{{#1},}
\renewbibmacro{in:}{}
\DeclareFieldFormat{pages}{#1}

\renewbibmacro*{volume+number+eid}{%
	\printfield{volume}%
	\iffieldundef{number}{}
	{\printtext[parens]{\printfield{number}}}
}

\renewbibmacro*{journal+issuetitle}{%
  \usebibmacro{journal}%
  \setunit*{\addspace}%
  \iffieldundef{series}
    {}
    {\newunit
     \printfield{series}%
     \setunit{\addspace}}%
  \usebibmacro{volume+number+eid}%
  \setunit{\bibpagespunct}%
  \printfield{pages}%
  \setunit{\addspace}%
  \usebibmacro{issue+date}%
  \setunit{\addcolon\space}%
  \usebibmacro{issue}%
  \newunit}
  
  \renewbibmacro*{note+pages}{%
  \printfield{note}%
  \newunit}

\newcommand{\R}{\mathbb{R}}
\newcommand{\Z}{\mathbb{Z}}
\newcommand{\N}{\mathbb{N}}
\renewcommand{\P}{\mathbb{P}}
\newcommand{\B}{\mathbb{B}}

\newcommand\rank{\mathrm{rank}}

\newcommand*{\F}{\mathbb{F}}
\newcommand*{\fg}{finitely generated}
\newcommand*{\grpgen}[1]{\left\langle{#1}\right\rangle}
\newcommand*{\ie}{i.e.\ }
\newcommand*{\inn}[1]{\iota_{#1}}
\newcommand*{\inv}[1]{{#1}^{-1}}

\newcommand*{\restr}[2]{{#1}|_{#2}}
\newcommand*{\restrb}[2]{\left({#1}\right)|_{#2}}
\newcommand*{\size}[1]{{\left|#1\right|}}
\newcommand*{\tf}{torsion-free}
\newcommand*{\ZmodZ}[1]{\Z / #1 \Z}
\newcommand*{\ZpZ}{\ZmodZ{p}}
\newcommand*{\eps}{\varepsilon}

\DeclareMathOperator{\End}{End}
\DeclareMathOperator{\GL}{GL}
\DeclareMathOperator{\Mat}{Mat}
\DeclareMathOperator\Fix{Fix}
\DeclareMathOperator{\Id}{Id}

\DeclareMathOperator\Stab{Stab}

\DeclareMathOperator\dc{dc}
\DeclareMathOperator\tdc{tdc}

\DeclareMathOperator\CR{cr}
\DeclareMathOperator\tcr{tcr}

\renewcommand{\phi}{\varphi}
\renewcommand{\epsilon}{\varepsilon}

\let\originalleft\left
\let\originalright\right
\renewcommand{\left}{\mathopen{}\mathclose\bgroup\originalleft}
\renewcommand{\right}{\aftergroup\egroup\originalright}

\renewcommand{\le}{\leqslant}
\renewcommand{\ge}{\geqslant}

\theoremstyle{plain}
\newtheorem{thm}{Theorem}[section]
\newtheorem{lem}[thm]{Lemma}
\newtheorem{prop}[thm]{Proposition}
\newtheorem{cor}[thm]{Corollary}

\newtheorem*{notation}{Notation}
\newtheorem{remark}[thm]{Remark}

\AddToHook{env/lem/begin}{\crefalias{thm}{lemma}}
\AddToHook{env/prop/begin}{\crefalias{thm}{proposition}}
\AddToHook{env/dt/begin}{\crefalias{thm}{definition and theorem}}
\AddToHook{env/cor/begin}{\crefalias{thm}{corollary}}

\newtheorem{quR}{Question}

\makeatletter
\newtheorem*{rep@theorem}{\rep@title}
\newcommand{\newreptheorem}[2]{%
\newenvironment{rep#1}[1]{%
 \def\rep@title{#2 \ref{##1}}%
 \begin{rep@theorem}}%
 {\end{rep@theorem}}}
\makeatother
\newreptheorem{theorem}{Theorem}

\newtheorem*{thm*}{Theorem}
\newtheorem*{lem*}{Lemma}
\newtheorem*{prop*}{Proposition}
\newtheorem*{cor*}{Corollary}
\newtheorem*{qu*}{Question}
\newtheorem*{dt*}{Definition and Theorem}
\newtheorem*{not*}{Notation}
\newtheorem*{exmp*}{Example}
\newtheorem*{exmps*}{Examples}
\newtheorem*{dprop*}{Definition and Proposition}
\newtheorem*{conj*}{Conjecture}

\theoremstyle{definition}
\newtheorem{defn}[thm]{Definition}
\newtheorem{exmp}[thm]{Example}

\newtheorem*{defn*}{Definition}
\newtheorem{conj}[thm]{Conjecture}

\AddToHook{env/defn/begin}{\crefalias{thm}{defn}}
\AddToHook{env/exmp/begin}{\crefalias{thm}{exmp}}
\AddToHook{env/qu/begin}{\crefalias{thm}{question}}
\AddToHook{env/conj/begin}{\crefalias{thm}{conjecture}}

\theoremstyle{plain}
\newtheorem{rem}[thm]{Remark}
\newtheorem*{rem*}{Remark}

\AddToHook{env/rem/begin}{\crefalias{thm}{remark}}

\crefname{prop}{Proposition}{Propositions}
\crefname{exmp}{Example}{Examples}
\crefname{cor}{Corollary}{Corollaries}
\crefname{lemma}{Lemma}{Lemmata}

\newcommand{\Address}{{
		\bigskip
		\small
		
		\textsc{Department of Mathematics, University of Manchester, UK \& \\ Heilbronn Institute for Mathematical Research, Bristol, UK}
		\par\nopagebreak
\textit{E-mail address}: \texttt{gemma.crowe@manchester.ac.uk} 
		
		\bigskip
		\par\nopagebreak
		\textsc{Institute for Mathematics, TU Berlin, Germany \& \\  Dept. of Mathematics, Heriot-Watt University \& \\ Maxwell Institute for Mathematical Sciences, Edinburgh} \par\nopagebreak
\textit{E-mail address}: \texttt{L.Ciobanu@hw.ac.uk, ciobanu@math.tu-berlin.de}

		\bigskip
		\par\nopagebreak
		\textit{E-mail address}: \texttt{pieter.senden@telenet.be}
		
		\bigskip
		\par\nopagebreak
		\textsc{Mathematical Institute, University of Oxford, UK} \par\nopagebreak
		\textit{E-mail address}: \texttt{corentin.bodart@maths.ox.ac.uk} 
}}

\title{Twisted commutativity and conjugacy ratio in groups}
\date{}

\author{Laura Ciobanu, Gemma Crowe and Pieter Senden\\ (with an Appendix by Corentin Bodart)}
\begin{document}

\maketitle
\begin{abstract}
In this paper we introduce and study the degree of twisted commutativity and the twisted conjugacy ratio of a finitely generated group $G$. The degree of twisted commutativity \(\tdc_X(\phi, G)\) generalises the degree of commutativity of $G$, by measuring the density of pairs of elements with trivial twisted commutators in the ball of radius $n$ of $G$, as $n \rightarrow \infty$, where the twisting is done with respect to an endomorphism $\phi$ of \(G\). We compute \(\tdc_X(\phi, G)\) for several classes of groups, including virtually abelian groups, groups of subexponential growth, and free groups.

We then study the twisted conjugacy ratio $\tcr_{X}(\phi, G)$, which is the limit at infinity of the quotient of the twisted conjugacy and standard growth functions. We 
compute $\tcr_{X}(\phi, G)$ for virtually abelian groups, and give examples of groups of exponential growth such that $\tcr_{X}(\phi, G) = 0$.\\


\end{abstract}
\unmarkedfntext{\emph{Keywords}: degree of commutativity, conjugacy ratio, twisted conjugacy, Reidemeister number}

\section{Introduction}
We introduce and study the generalisations of two group concepts: firstly, we expand the degree of commutativity, defined in \cite{Gustafson} for finite groups and subsequently for infinite groups in \cite{dcA}, and secondly, we extend the notion of conjugacy ratio defined in \cite{CCM19}. Both of these concepts have received much attention in the last decade with respect to infinite groups \cite{Cox3, Guo25, HKZ24, Toiton, Val19}.

Our generalisation allows for `twisting' by an endomorphism, and leads to the concepts of \emph{degree of twisted commutativity} $\tdc$, and \emph{twisted conjugacy ratio} $\tcr$. In this paper we remark that the two quantities, $\tdc$ and $\tcr$, are equal in any finite group for the same endomorphism, and we explore the connections between them in infinite groups.

The degree of twisted commutativity of finitely generated groups measures the density of the pairs of elements with trivial twisted commutators in the ball of radius \(n\) of the group, as \(n\rightarrow\infty\). That is, we quantify how many pairs $(x,y)$ of elements `almost commute', where `almost commuting' means $xy=y\phi(x)$ for some fixed endomorphism $\phi$. More precisely, let \(G\) be a group generated by a finite set \(X\), let \(\B_{G,X}(n)\) denote the ball of radius \(n\) with respect to \(X\), and let \(\phi\) be an endomorphism of \(G\). The \emph{degree of twisted commutativity} \(\tdc_X(\phi, G)\) of \(G\), with respect to \(X\) and the `twist' \(\phi\), is then
\begin{align}\label{def:tdc}
\tdc_X(\phi,G)\coloneq\limsup_{n\rightarrow\infty}\frac{|\{(x, y) \in \B_{G,X}(n)^2 :  xy=y\phi(x)\}|}{|\B_{G,X}(n)|^2}.
\end{align}

For a finite group \(F\) and an endomorphism \(\phi\colon F \rightarrow F\), we define the \(\tdc(\phi,F)\) as
\begin{align}\label{tdcfinite}
	\tdc(\phi,F)\coloneq\frac{|\{ (x,y) \in F \times F \ : \ xy=y\phi(x) \}|}{|F|^2}.
\end{align}

There is a remarkable link, for finite $F$, between \(\tdc(\phi,F)\) and the \emph{Reidemeister number} \(R(\phi)\) of \(\phi\) (see \cref{prop:tdcFiniteGroupsReidemeisterNumber}), where \(R(\phi)\) is the number of \(\phi\)-conjugacy classes in $F$:
\begin{align}\label{tdcfiniteR}
	\tdc(\phi,F)=\frac{R(\phi)}{|F|}.
\end{align}
Here, for any group $G$, and any endomorphism $\phi$, two elements \(x, y \in G\) are \(\phi\)-\emph{conjugate}, and belong to the same $\phi$-\emph{conjugacy class}, if there exists a \(z \in G\) such that \(x = zy \inv{\phi(z)}\).

As previously mentioned, the degree of twisted commutativity extends the degree of commutativity \(\dc_X(G)\) of a group \(G\), that is, where \(\phi=\Id_{G}\) in \eqref{def:tdc}. For a finite group \(F\), it is a celebrated result by Gustafson \cite{Gustafson} that \(\dc(F)=1\) exactly if \(F\) is abelian, and \(\dc(F) \le 5/8\) otherwise. In fact, in any finite group \(F\), \(\dc(F)=\frac{|C(F)|}{|F|}\), where $|C(F)|$ is the number of conjugacy classes in \(F\), and since $R(\phi) \le |C(F)|$, we get $\tdc(\phi,F)\le\dc(F)$ for any endomorphism $\phi$ of $F$ (see \cref{cor:tdcAtMostDc}).


Our starting point is the following conjecture, as well as the similarities and differences between the twisted and `untwisted' degrees of commutativity.

\begin{conj}\cite[Conj. 1.6]{dcA} \label{maincon}
	Let \(G\) be a group generated by a finite set \(X\). Then \(\dc_X(G) > 0\) if and only if \(G\) is virtually abelian.
\end{conj}

Although several results about the degree of twisted commutativity are similar to those for the degree of commutativity, the map \(\phi\) plays a very important role and gives sometimes different behaviours of the two quantities, as we will see in this paper. The main class of groups we investigate (see \cref{subexp}) are those of stable subexponential growth (see \cref{stable}). 
This includes all finitely generated virtually nilpotent groups. Since all finitely generated virtually nilpotent groups are residually finite, the theorem below holds for all virtually nilpotent groups.

\begin{reptheorem}{theo:tdcIsZeroForFGResFiniteGroupSubexpGrowthNonVirtuallyAbelian}
Let $G$ be a finitely generated residually finite group of stable subexponential growth that is not virtually abelian, and let $X$ be a finite generating set. Then \(\tdc_X(\phi,G)=0\) for any endomorphism $\phi$ of $G$. 
\end{reptheorem}

We then consider virtually abelian groups and show that \(\tdc_{X}(\phi, G)\) is intricately connected to the twist \(\phi\) and its action on the finite index free abelian subgroups of \(G\).

\begin{reptheorem}{thm:CharacterisationTdcNonZeroVirtuallyAbelian}
	Let \(G\) be an infinite, \fg{}, virtually abelian group with finite generating set \(X\). Let \(A\) be a \tf{}, fully invariant, abelian subgroup of finite index in \(G\). Let \(\phi \in \End(G)\) be an endomorphism and let \(\inn{g}\) denote conjugation by \(g\). Then the following hold:
	\begin{itemize}
	
	\item[(i)] \(\tdc_{X}(\phi, G) > 0\) if and only if there exists \(g \in G\) such that \(\restrb{\inn{g} \circ \phi}{A} = \Id_{A}\).
	
	\item[(ii)] For any finite generating set \(Y\) of \(G\), \(\tdc_{X}(\phi, G) = 0\) if and only if \(\tdc_{Y}(\phi, G) = 0\).
	\end{itemize}
\end{reptheorem}
\noindent Theorem \ref{thm:va_indep} in the Appendix strengthens Theorem \ref{thm:CharacterisationTdcNonZeroVirtuallyAbelian}, as it shows that $\tdc$ for virtually abelian groups is independent of generating set.

\medskip

We also link the degree of twisted commutativity of a virtually abelian group to the degrees of twisted commutativity of its finite quotients.
For a group \(G\) and a subgroup \(H\), we write \(H \le_{f} G\) if \(H\) has finite index in \(G\).

\begin{reptheorem}{theo:InfimumTdcFiniteQuotientsVirtuallyAbelian}
	Let \(G\) be an infinite, \fg{} virtually abelian group and let \(A\) be a \tf{} abelian fully invariant subgroup of finite index in \(G\). Let \(\phi \in \End(G)\). Suppose that \(\restrb{\inn{g} \circ \phi}{A} \ne \Id_{A}\) for all \(g \in G\). Then
	\[
		\inf\{\tdc(\overline{\phi}, G / N) : N \lhd_{f} G, \phi(N) \le N, \overline{\phi} \text{ induced endomorphism on } G / N\} = 0.
	\]

	This infimum is also equal to \(\tdc_{X}(\phi, G)\) for any finite generating set \(X\) of \(G\).
\end{reptheorem}

In the realm of groups of exponential growth, we cover the case of free groups:
\begin{reptheorem}{thm:free}
	Let $\F=\F(X)$ be a free group of finite rank $m\ge 2$ and basis $X$, and let $\phi \in \End(\F)$. Then $\tdc_X(\phi,\F)=0$.
\end{reptheorem}
The proof for free groups relies on analysing the growth rates of fixed subgroups of free groups (see \cref{prop:free}) and we expect our results on growth rates to generalise to further groups of exponential growth where the fixed subgroups are understood.

The second goal of the paper is to generalise the notion of `conjugacy ratio', which compares the number of conjugacy classes to the number of elements in a suitable way (see \cite{CCM19}), to the notion of \emph{twisted conjugacy ratio}. For a group $G$ with finite generating set $X$ and for \(\phi \in \mathrm{End}(G)\), we let \(C^{\phi}_{G,X}(n)\) denote the set of \(\phi\)-conjugacy classes of \(G\) which have a representative in \(\mathbb{B}_{G,X}(n)\). The \emph{\(\phi\)-twisted conjugacy ratio} of \(G\) with respect to \(X\) is defined as 
\begin{equation}
	\tcr_{X}(\phi, G) \coloneq \limsup\limits_{n\rightarrow \infty} \frac{|C^{\phi}_{G,X}(n)|}{|\mathbb{B}_{G,X}(n)|}.
\end{equation}
Analogous to the equality in the `untwisted' version, the degree of twisted commutativity and the twisted conjugacy ratio are equal in finite groups (see \cref{rmk:finite groups tcr tdc}): for any finite group $F$, \(\tcr_{X}(\phi, F)\) does not depend on \(X\) and
\begin{equation}
	\tcr(\phi, F) = \frac{R(\phi)}{\size{F}} = \tdc(\phi, F).
\end{equation}

We first consider groups of stable subexponential growth, and show a twisted adaptation of \cite[Proposition 3.4]{CCM19} for virtually abelian groups.

\begin{reptheorem}{prop:tcrNonZeroVirtuallyAbelianGroups}
    Let \(G\) be a finitely generated virtually abelian group and \(A\) an abelian normal subgroup such that \([G : A] = m\).
    Let \(\phi \in \End(G)\) be such that \(\phi(A) \le A\) and \(X\) be a generating set of \(G\).
    Suppose there exists \(g \in G\) such that \(\restrb{\inn{g} \circ \phi}{A} = \Id_{A}\).
    Then \(\tcr_{X}(\phi, G) > 0\).
    
    More precisely, if \(g_{1}, \ldots, g_{m}\) are coset representatives of \(G / A\) and there exists an integer \(k \ge 1\) such that 
    \begin{enumerate}[(i)]
    	\item \(\restrb{\inn{g_{i}} \circ \phi}{A} = \Id_{A}\) for \(i \in \{1, \ldots, k\}\), and
	\item \([g_{i}A]_{\bar{\phi}} \ne [g_{j}A]_{\bar{\phi}}\) for \(i, j \in \{1, \ldots, k\}\) with \(i \ne j\), where \(\bar{\phi}\colon G / A \to G / A: gA \mapsto \phi(g)A\),
    \end{enumerate}
    then \(\mathrm{tcr}_{X}(\phi, G) \ge \frac{k}{m^{2}}\). 
\end{reptheorem}

\noindent Theorem \ref{thm:va_tcr_indep} in the Appendix strengthens Theorem \ref{prop:tcrNonZeroVirtuallyAbelianGroups}, as it shows that $\tcr$ for virtually abelian groups is independent of the generating set.


We also consider examples of groups with exponential growth, such that the twisted conjugacy ratio equals zero.

\begin{reptheorem}{thm:hyp group tcr 0}
    Let $G$ be a hyperbolic group, and let $\phi \in \mathrm{Aut}(G)$ be of finite order. Then $\tcr_{X}(\phi, G) = 0$ for any finite generating set $X$ of $G$. 
\end{reptheorem}

We show that $\tcr_{X}(\phi, G) = 0$ also holds for certain Artin groups $G$ and selected automorphisms $\phi$, as in \cref{prop:DA} and \cref{exmp:RAAG tcr 0}.
 
Finally, while this paper introduces the two concepts in the title and provides computations for several classes of groups, there are many research avenues widely open. Our initial results lead to open questions, conjectures and invitations to generalisations, which we collect in \cref{sec:reflections}.  An Appendix by Corentin Bodart, added after the first version of the paper was made public, shows that, for virtually abelian groups, the values of $\tdc$ and $\tcr$ do not depend on the generating set, but are group invariants.

\section{Preliminaries}
We start by recalling and proving the necessary concepts on twisted conjugacy, Reidemeister numbers and fixed subgroups, and by establishing connections between the Reidemeister number of \(\phi\) and the \(\tdc(\phi, F)\) for a finite group \(F\). We also set up the basic notation for growth in infinite groups.

\subsection{Twisted conjugacy, twisted commutativity, and fixed subgroups}
Let \(G\) be a group and \(\phi \in \End(G)\). Two elements \(x, y \in G\) are \emph{\(\phi\)-conjugate} if \(x = zy\inv{\phi(z)}\) for some \(z \in G\). The \(\phi\)-conjugacy class of \(x\) is denoted by \([x]_{\phi}\). The \emph{Reidemeister number} of \(\phi\) is the (finite or infinite) number of \(\phi\)-conjugacy classes and is denoted by \(R(\phi)\). Given \(g \in G\), the \emph{\(\phi\)-twisted stabiliser} of \(g\) is the subgroup
	\[
		\Stab_{\phi}(g) \coloneq \{x \in G : xg\inv{\phi(x)} = g\}.
	\]
We write \(\inn{g}: G \to G: x \mapsto gx\inv{g}\) for the inner automorphism (or conjugation) associated to \(g\), and if \(H \le G\) is a subgroup such that \(\phi(H) \le H\), we write \(\restr{\phi}{H}\) for the induced endomorphism on \(H\).

The subgroup of fixed points of \(\phi\), or `fixed subgroup of \(\phi\)',  is defined as
\begin{equation}\label{def:fixed}
	\Fix(\phi)\coloneq\{g\in G : \phi(g)=g\},
\end{equation}
and fixed subgroups are linked to  \(\phi\)-twisted stabilisers via conjugation:
\begin{lem}[{\cite[Lemma~T.1.10]{Senden23}}]\label{lem:fixstab}
	Let \(G\) be a group, \(\phi \in \End(G)\) and \(g \in G\). Then \(\Stab_{\phi}(g) = \Fix(\inn{g} \circ \phi)\).
\end{lem}

\cref{lem:fixstab} can help us rewrite the numerator in the definition of \(\tdc_{X}(\phi, G)\):
\begin{align}
	\size{\{(x, y) \in (\B_{G, X}(n))^{2} : xy = y\phi(x)\}}	&= \sum_{y \in \B_{G,X}(n)} \size{\{x \in \B_{G, X}(n) : xy = y\phi(x)\}}	\nonumber\\
														&= \sum_{y \in \B_{G, X}(n)} \size{\Stab_{\phi}(y) \cap \B_{G, X}(n)}	\label{eq:tdcInStabilisers}\\
														&= \sum_{y \in \B_{G, X}(n)} \size{\Fix(\inn{y} \circ \phi) \cap \B_{G, X}(n)}.	\label{eq:tdcInFixedPoints}
\end{align}
Thus, the definition of \(\tdc_X(\phi,G)\) expands to:
\begin{align}	\label{def:tdcsum}
	\tdc_X(\phi,G)=\limsup_{n\rightarrow\infty}\frac{\sum_{g\in \B_{G,X}(n)}| \Fix(\inn{g} \circ \phi)\cap \B_{G,X}(n)|}{|\B_{G,X}(n)|^2}.
\end{align}
For finite groups, there are immediate connections between the concepts defined above.
\begin{prop}	\label{prop:tdcFiniteGroupsReidemeisterNumber}
	Let \(F\) be a finite group and \(\phi \in \End(F)\). Then
	\[
		\tdc(\phi, F) = \frac{R(\phi)}{\size{F}}.
	\]
\end{prop}
\begin{proof}
	By the orbit-stabiliser theorem, we have that
\[
	R(\phi)	=	\sum_{g \in F} \frac{1}{\size{[g]_{\phi}}} =	\frac{1}{\size{F}} \sum_{g \in F} \size{\Stab_{\phi}(g)}.
\]
Since \(F\) is a finite group, the summation on the right-hand side equals \(\size{F}^{2} \cdot \tdc(\phi, F)\) by \eqref{eq:tdcInStabilisers}. The result then follows.
\end{proof}

\begin{prop}[{\cite[Theorem~5]{FelshtynHill94}}]	\label{prop:ReidemeisterNumberEqualsFixedConjugacyClasses}
	Let \(F\) be a finite group and \(\phi \in \End(F)\). Then \(R(\phi)\) equals the number of ordinary conjugacy classes that are fixed by \(\phi\), \ie the number of conjugacy classes \([g]\) such that \([g] = [\phi(g)]\).
	
	In particular, \(R(\phi)\) is at most the number of conjugacy classes of \(F\).
\end{prop}

\begin{cor}	\label{cor:tdcAtMostDc}
	Let \(F\) be a finite group and \(\phi \in \End(F)\). Then \(\tdc(\phi, F) \le \dc(F)\).	
\end{cor}
\begin{proof}
	Since \(\tdc(\Id, F) = \dc(F)\) and \cref{prop:ReidemeisterNumberEqualsFixedConjugacyClasses} implies $R(\phi)\le R(\Id)$, we have
	\begin{equation*}\label{tdc-def}
		\tdc(\phi, F) = \frac{R(\phi)}{\size{F}} \le \frac{R(\Id)}{\size{F}} = \tdc(\Id, F) = \dc(F). \qedhere
	\end{equation*}
\end{proof}

We also state the famous result by Gustafson \cite{Gustafson} for reference purposes:
\begin{thm}	\label{theo:GustafsonsBound}
	Let \(F\) be a finite group.
	If \(F\) is non-abelian, then \(\dc(F) \le \frac{5}{8}\).
\end{thm}

\subsection{Groups with unique roots}
Several results in this paper rely on the `unique roots' property.
We say a group $G$ has \emph{unique roots} if $g^n=h^n$ for $g,h \in G$ and $n\ne 0$ implies $g=h$; such a group is also called an \emph{$R$-group} in the literature. Any torsion-free group with cyclic centralizers of non-trivial elements has the unique root property, so torsion-free hyperbolic groups have unique roots. On the other side of the spectrum, torsion-free abelian groups have unique roots as well. More generally, bi-orderable groups have unique roots.

\begin{lem}\label{lem:FixedPointInfiniteIndexTFAbelianGroup}
Let $G$ be a group with unique roots. Then for any endomorphism $\phi \ne \Id$ of $G$, the subgroup $\Fix(\phi)$ has infinite index in $G$.
\end{lem}
\begin{proof}
	We prove the contrapositive: if \(\Fix(\phi)\) has finite index in \(G\), then \(\phi = \Id\).
	Let \(g \in G\) be arbitrary. Since \(\Fix(\phi)\) has finite index in \(G\), there is a \(k \ne 0\) in \(\Z\) such that \(g^{k} \in \Fix(\phi)\). Hence, \(\phi(g)^{k} = g^{k}\). Since roots are unique in $G$, it follows that \(\phi(g) = g\). As \(g\) was arbitrary, we conclude that \(\phi = \Id\).
\end{proof}

\subsection{Growth functions and growth rates}
Throughout the paper, \(f(n) \sim g(n)\) will signify \(f(n)/g(n) \rightarrow 1\) as \(n \rightarrow \infty\).
For a finitely generated group \(G\) with generating set \(X\), the \emph{exponential growth rate} of \(G\) with respect to \(X\) is
\begin{equation}\label{growth_rate}
\gamma_X(G): = \limsup_{n\rightarrow\infty} \sqrt[n]{ |\B_{G,X}(n)|}=\lim_{n\rightarrow\infty} \sqrt[n]{ |\B_{G,X}(n)|},
\end{equation}
where the last equality is a classical result.
\begin{defn}
	A finitely generated group \(G\) with generating set \(X\) is said to have \emph{exponential growth} if \(\gamma_X(G) > 1\) and \emph{subexponential growth} if \(\gamma_X(G) = 1\).
	This dichotomy does not depend on the generating set \(X\).
\end{defn}
For any \(\epsilon > 0\), if \(\gamma = \gamma_X(G)>1\), then for sufficiently large \(n\), 
\begin{equation}	\label{eq_approx}
	\gamma^n \le |\B_{G,X}(n)| \le (\gamma + \epsilon)^n,
\end{equation}
and we write $|\B_{G,X}(n)| \sim \gamma^n.$
If we replace balls with spheres, we get the same limit in \eqref{growth_rate} and inequality in \eqref{eq_approx}.

For a free group $\F$ of finite rank $m\ge 2$, the exponential growth rate is $\gamma_X(\F)=2m-1$ and the growth function of $\F$ is $|\B_{\F,X}(n)| \sim (2m-1)^n.$

\section{Free groups}\label{sec:exp}

In this section we prove, in \cref{thm:free}, that in any finitely generated free group the degree of twisted commutativity $\tdc$ is $0$ for any endomorphism. To prove this, we find upper bounds for the relative growth rates of fixed subgroups (as in (\ref{def:fixed})) in free groups. This will be done with the help of \cref{prop:free}.

Let $G$ be a group with finite generating set $X$. For a subgroup $H \le G$, we define the \emph{relative growth function} of $H$ in $G$ to be
\[
	\gamma_{H,X}(n)\coloneq|H \cap \B_{G,X}(n)|
\]
for any $n\ge 0$. 
The \emph{relative exponential growth rate} of $H$ in \(G\) is then  
\begin{equation}	\label{sbgp_growth_rate}
	\gamma_X(H): = \limsup_{n\rightarrow\infty} \sqrt[n]{\gamma_{H,X}(n)},
\end{equation}
and is known to be a limit if $G$ is a free group (\cite[Theorem 1.5]{Ol17}).
Furthermore, if $H$ has finite index in $G$, then  $\gamma_X(H) = \gamma_X(G).$

\begin{remark}
For any infinite index subgroup $H$ of a free group $G$ (or more general groups $G$, see for example \cite[Theorem 1.1]{DFW19}), it is known that $\gamma_X(H)<\gamma_X(G)$. However, there exist sequences of infinite index subgroups for which the limit of their relative growth rates is equal to $\gamma_X(G)$ (\cite[Theorem 9.4]{DFW19}).

To obtain our main result, we need to ensure that there is a universal gap between the growth rates of the fixed subgroups and the growth rate of the free group, which we show in 
\cref{prop:free}.
\end{remark}

\begin{prop}	\label{prop:free}
	Let $\F=\F(X)$ be a free group of finite rank $m\ge 2$, and let $k\ge 1$ be a fixed integer.
	Then the following inequality holds:
	\[
		\max_{\rank(K)\le k,  \ K  \underset{i.i}{\le}   \F}\gamma_X(K)< \gamma_X(\F)=2m-1,
	\]
	where \(K \underset{i.i}{\le} \F\) means that \(K\) has infinite index in \(\F\).

	That is, there exists a constant $\gamma_k<2m-1$ such that any subgroup $K \le \F$ of infinite index and rank $\le k$ satisfies $\gamma_X(K) \le \gamma_k$.
\end{prop}

To prove the proposition we use `core' or `Stalling graphs' for finitely generated subgroups of free groups \cite[Prop. 3.8]{KM02}, which we recall below.

Let \(\F = \F(X)\) be as above, free of rank \(m \ge 1\) with basis \(X\), and let $H$ be a finitely generated subgroup.
The \emph{core graph} of \(H\), denoted \(C_{H}\), is a finite, connected, graph with the following properties:
\begin{enumerate}[(i)]
	\item It is \(X\)-labelled and oriented. That is, any oriented edge $(v_1,v_2)$ has some label $a \in X$, which
	 means that the edge $(v_2, v_1)$ can be crossed and has label $a^{-1}$, but $(v_2, v_1)$ and $a^{-1}$ are not explicitly drawn.
	\item It is \emph{folded}: that is, no two edges incident to a vertex have the same \(X\)-label.
	\item It contains a distinguished vertex \(\circ\).
	\item The labels of all the closed paths (or loops) at \(\circ\) give exactly \(H\), and the graph \(C_{H}\) is minimal with this property.
\end{enumerate}
All vertices in $C_H$ have degree $\ge 2$ (the degree is the number of all edges incident to a vertex, whether incoming or outgoing), except possibly $\circ$;
and every vertex is the initial vertex (and the terminal vertex) of at most $m$ edges, labelled by pairwise different letters in $X$. In \cref{two_graphs}, the graph $C_H$ on the right is the core graph of the subgroup $H=\langle a, b^2a\rangle$ of the free group on $X=\{a,b\}.$

\begin{figure}[h] 
	\begin{center}
		\includegraphics[scale=0.75]{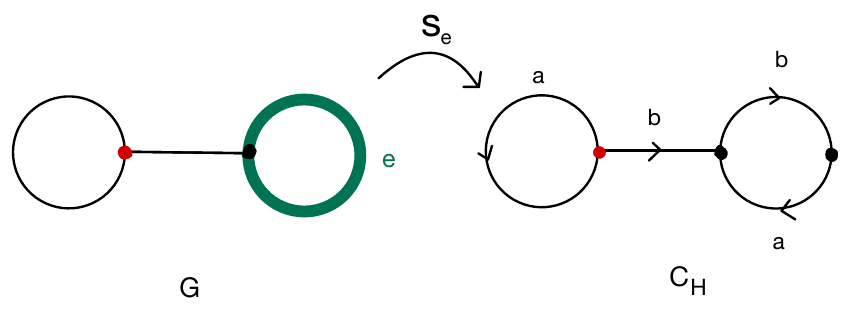}
		\caption{Topological graph $G$ gives $u(C_H)$ (with $\circ$ in red) by the subdivision of $e$}
	\label{two_graphs}
	\end{center}
\end{figure}

The following are standard facts about core graphs:

\begin{lem}[{\cite[Prop.~8.3]{KM02}}]	\label{lem_deg}
	Let $H$ be a finitely generated subgroup of a free group $\F$ of rank $m>1$, and let $C_H$ be its core graph. 
	Then $H$ has finite index if and only if $C_H$ has all vertex degrees equal to $2m$.
	In particular, if $C_H$ has a vertex of degree $2$, then $H$ has infinite index in $\F$. 
\end{lem}

We will also need the unlabelled and undirected version of $C_H$, denoted by $u(C_H)$, and some general facts about spectral radii from graph theory.

\begin{defn}
	The \emph{spectral radius} (or \emph{entropy}) of a graph $\Gamma$ is denoted by $\rho(\Gamma)$ and is the growth rate of the set of (unlabelled) closed paths at some basepoint $\circ$ in $\Gamma$. 
	That is, $\rho(\Gamma)=\lim_{n\rightarrow\infty} \sqrt[n]{\gamma_\circ(n)}$, where $\gamma_\circ(n)$ is the number of closed paths of length $n$ at $\circ$.
\end{defn}

It is well known that $\rho(\Gamma)$ is the spectral radius, or largest eigenvalue, of the adjacency matrix of $\Gamma$.
When a directed, labelled, graph is deterministic (from every vertex there is at most one edge labelled by some generator $x \in X$) and ergodic (one can reach any vertex from any other vertex by a directed path), its entropy is equal to the growth rate for the set of words labelling the loops at some base point. Specifically, since the core graph $C_H$ of a subgroup $H$ of $\F$ is, by definition, deterministic and ergodic, the growth rate $\gamma_X(H)$ of $H$ is equal to the entropy of $C_H$ (see for example \cite{Ol17}), that is, 
\begin{equation}	\label{eq_rates}
	\gamma_X(H)=\rho(C_H).
\end{equation}

In other words, counting elements in a subgroup $H$ amounts to counting loops (without backtracking) at the vertex $\circ$ in its core graph $C_H$.
For further use, we also record the fact that if a finite connected graph $\Gamma$ has all vertex degrees $\le D$, then
\begin{equation}	\label{max_rates}
	\rho(\Gamma)\le D-1,
\end{equation}
with equality only when all degrees are equal to $D$.

To quantitatively study the relative growth rates of subgroups via their core graphs, we need to introduce one more type of graphs, which we call \emph{topological graphs}. 
\begin{defn}\label{def:topgraph}
	A {\it topological graph} of rank $k\ge 2$ is a finite, unlabelled graph with a distinguished vertex $\circ$ in which all vertices have degree $\ge 3$ 
	and whose fundamental group is the free group of rank $k$.
\end{defn}

\begin{figure}[h]
	\begin{center}
		\includegraphics[scale=0.12]{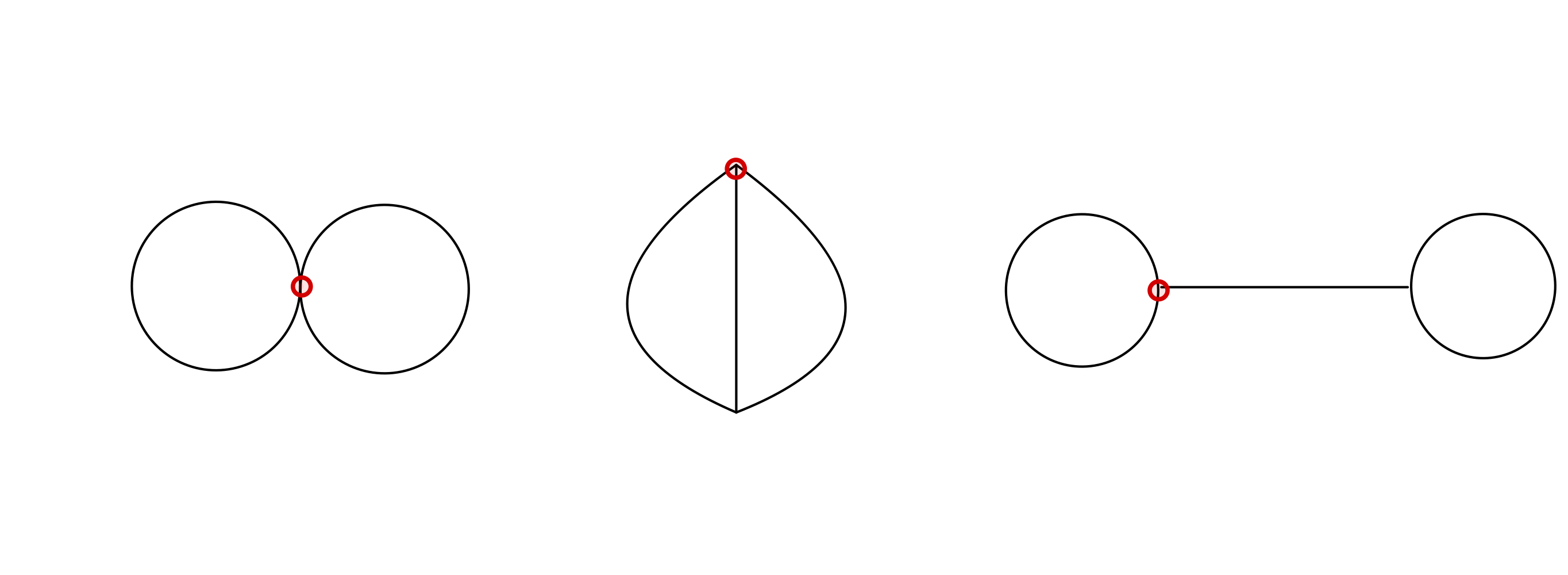}
		\caption{The set $Top(2)$ consists of $3$ graphs}\label{three_graphs}
	\end{center} 
\end{figure}

\begin{rem}\label{rem:top_graph_v2} We list some properties of topological graphs and their relations with the (core graphs of) subgroups of free groups.
\begin{enumerate}[(i)]
	\item Let $Top(k)$ be the set of all non-isomorphic topological graphs of rank $k$, for some $k\ge 1$.  
	Then the set $Top(k)$ is finite (see \cref{three_graphs} for $Top(2)$). 
\end{enumerate}
Next, let $\F$ be a free group of rank $m>1$. Let \(H\) be a finitely generated subgroup of \(\F\).
Suppose the core graph $C_H$ has no vertex of degree $1$.
\begin{enumerate}[(i)] \label{core-graph}
	\setcounter{enumi}{1}
	\item	We can associate a topological graph $\mathcal{T}_H$ to the subgroup $H$, obtained from the unlabelled $u(C_H)$, by removing all vertices of $u(C_H)$ of degree 2 (except $\circ$) and connecting the edges incident to those vertices to make `longer' edges. Then $\pi_1(\mathcal{T}_H) \cong H$.
	\item Conversely, if \(H\) has rank \(k\), we can obtain $u(C_H)$ by (repeatedly) subdividing the edges of one of the graphs in $Top(k)$ (see \cref{two_graphs}).
	\item Let $\Gamma$ be a finite, connected, graph with some base vertex $\circ$ such that no vertex degree is $>2m$ or $=1$, the latter except for the vertex $\circ$. Then there exists a finitely generated subgroup $K$ of $\F$ such that $\Gamma=C_K$ by assigning edge labels and directions appropriately (see \cite{J02}). 
\end{enumerate}
\end{rem}


We will use a classical result on spectral radii of graphs called Hoffman and Smith's Subdivision Theorem. An \emph{internal path} of a graph is a sequence of adjacent vertices, all with vertex degrees $2$, except for the two end-vertices; the end-vertices each have degree strictly greater than $2$, and may possibly be the same vertex.

For any graph $\Gamma$, we write $s_e(\Gamma)$ for the graph obtained by subdividing some (single) edge $e$ of $\Gamma$ that lies on an internal path, as in \cref{two_graphs}.

\begin{thm}[{\cite{HS74}}]		\label{thm:subdiv}
	For a graph $\Gamma$ with an internal path $\nu$, if $s_e(\Gamma)$ is the graph obtained by subdividing an edge $e$ on $\nu$, then the spectral radii satisfy $\rho(s_e(\Gamma))<\rho(\Gamma)$.
\end{thm}

\begin{proof}[Proof of \cref{prop:free}]
Let $H$ be a subgroup of rank $k$ with core graph $C_H$. We choose $H$ so that the distinguished vertex of $C_H$ is not of degree $1$ by picking a conjugate of $H$, if necessary; this will not affect the growth rate.

By \cref{rem:top_graph_v2} (iii), $u(C_H)$ can be obtained by (repeatedly) subdividing edges in some graph $\mathcal{T}_H \in Top(k)$, and with each subdivision the entropy strictly decreases by \cref{thm:subdiv}. Since $H$ has infinite rank, $C_H$ has some vertex degrees different from $2m$ by \cref{lem_deg}. Hence, either
\begin{itemize}
\item[(a)] $\mathcal{T}_H$ has some vertex degrees different from $2m$, or
\item[(b)] all vertex degrees in $\mathcal{T}_H$ are equal to $2m$, and 
$u(C_H)$ is equal to $s_{e_p}(\cdots (s_{e_1}(\mathcal{T}_H)))$, where the $e_i$, $1\le i \le p$, are edges on some internal paths; that is, the unlabelled core graph of $H$ is obtained from $\mathcal{T}_H$ by a non-empty sequence of subdivisions.
\end{itemize}
 In case (a), $\gamma_X(H)=\rho(C_H)\le \rho(\mathcal{T}_H)< 2m-1=\gamma_X(\F)$; the first inequality is an equality if $u(C_H)=\mathcal{T}_H$ and is a strict inequality otherwise, by the subdivision operation effect on the entropy (\cref{thm:subdiv}). The second inequality holds by \eqref{max_rates}.

 In case (b), we get the inequalities $\gamma_X(H)=\rho(C_H)\le \rho(s_{e_1}(\mathcal{T}_H))<\rho(\mathcal{T}_H)= 2m-1$ from \cref{thm:subdiv}; the last equality holds by \eqref{max_rates}.

Let $\mathcal{G}_k$ be the set of finite graphs consisting of the topological graphs in (a) and those obtained by a single edge-subdivision of a topological graph as in (b); for the latter, we mean the graphs of the form $s_e(\mathcal{T})$ for some $\mathcal{T} \in Top(k)$ and $e\in E(\mathcal{T})$, where all vertex degrees in $\mathcal{T}$ are $2m$. The set $\mathcal{G}_k$ is finite because by \cref{rem:top_graph_v2}(i) the set $Top(k)$ is finite, and for each graph in $Top(k)$ there are only finitely many edges to subdivide in a single edge-subdivision. Then the maxima among growth rates of 
subgroups and of entropies of graphs exist and satisfy
\[
	\max_{\rank(K)=k, \ K \underset{i.i}{\le} \F} \gamma_X(K) = \max_{\Gamma\in \mathcal{G}_k} \rho(\Gamma) <2m-1.
\]
The equality holds because the maximum relative growth rate of a subgroup of rank $k$ is achieved for one of the subgroups with core graph stemming from $\mathcal{G}_k$ by Remark \ref{core-graph}(iv), and the inequality since $\rho(\Gamma)<2m-1$ for each of the finitely many $\Gamma \in \mathcal{G}_k$. We get the result by taking the union of all subgroups of rank $j\le k$ and applying the above for each rank.
\end{proof}

\begin{thm}	\label{thm:free}
	Let $\F=\F(X)$ be a free group of finite rank $m\ge 2$ with basis $X$, and let $\phi \in \End(\F)$. Then $\tdc_X(\phi,\F)=0$.
\end{thm}
\begin{proof}
For a fixed $\phi \in \End(\F)$, we use \eqref{def:tdcsum} as the definition of $\tdc_X(\phi,\F)$, and with that notation, we write $H_g \coloneq \Fix(\inn{g} \circ \phi)$. 
By the Bestvina-Handel Theorem (\cite{BH92}), the rank of any fixed subgroup of an endomorphism of a free group is bounded by the rank of the ambient free group, so $\rank(H_g)\le m$ for all $g \in \F$.

Let $\lambda=\gamma_m$ be as in \cref{prop:free}, so $\lambda < 2m-1$. Then for any $\epsilon>0$ and large enough $n$, the relative growth function of $H$ satisfies $|H \cap \B_{\F,X}(n)| \le (\lambda +\epsilon)^n$. So, 
 in the summation \eqref{def:tdcsum}, we get
 \begin{align*}\label{tdcsum}
\tdc_X(\phi,\F) &= \limsup_{n\rightarrow\infty}\frac{\sum_{g\in \B_{\F,X}(n)}| H_g\cap \B_{\F,X}(n)|}{|\B_{\F,X}(n)|^2}\\
&\le \limsup_{n\rightarrow\infty}\frac{|\B_{\F,X}(n)| (\lambda +\epsilon)^n}{|\B_{\F,X}(n)|^2}\\
&=\limsup_{n\rightarrow\infty}\frac{(\lambda + \epsilon)^n}{|\B_{\F,X}(n)|}.
\end{align*}
We can pick $\epsilon$ so that $\lambda+\epsilon < 2m-1$, and since $|\B_{\F,X}(n)| \ge (2m-1)^n$, we thus get
\[
	\limsup_{n\rightarrow\infty}\frac{(\lambda + \epsilon)^n}{|\B_{\F,X}(n)|} \le \limsup_{n\rightarrow\infty}\frac{(\lambda + \epsilon)^n}{ (2m-1)^n}=0,
\]
which proves the theorem.
\end{proof}

\section{Twisted commutativity and subexponential growth}\label{subexp}

\begin{defn}	\label{stable}
	A finitely generated group \(G\) with generating set \(X\) is said to be of {\em stable subexponential growth} if \(\lim\limits_{n\rightarrow\infty}\frac{|\B_{G,X}(n+1)|}{|\B_{G,X}(n)|}=1\).
\end{defn}
Note that being of stable subexponential growth implies that \(\gamma_X(G) = 1\), and hence that the group has subexponential growth.

\medskip

By the celebrated result of Gromov, every finitely generated group of polynomial growth -- where \(\B_{G, X}(n)\) is bounded above by a polynomial function -- is virtually nilpotent. All these groups are of stable subexponential growth by results of Pansu (see \cite{P83},\cite[Corollary 9]{BD13}, and also \cite{bass}): if \(G\) is a finitely generated, virtually nilpotent group, and \(X\) is any generating set, then, for some exponent \(d\) and constant $c$,
\begin{equation}	\label{eqn:pol_bounds}
	|\B_{G,X}(n)| \sim c n^d.
\end{equation}
The exponent \(d\) is calculated explicitly in \cite{bass}; for a virtually abelian group it is equal to the rank of a finite index free abelian subgroup (unique up to isomorphism).

From \eqref{eqn:pol_bounds}, we get that for any positive integer \(k\), 
\begin{equation}	\label{eqn:ball_ratio}
	\lim_{n \to \infty} \frac{|\B_{G,X}(n+k)|}{|\B_{G,X}(n)|} = 1.
\end{equation}

The main result we require for this class of groups is the following:

\begin{prop}[{\cite[Proposition~2.2]{indexformula}}] \label{pepprop}
	Let \(G\) be a finitely generated group with stable subexponential growth and let \(X\) be a generating set. For every finite index subgroup \(H\le G\) and every \(g \in G\), we have
	\[
		\lim_{n\rightarrow\infty}\frac{|gH\cap\B_{G,X}(n)|}{|\B_{G,X}(n)|}=\lim_{n\rightarrow\infty}\frac{|Hg\cap\B_{G,X}(n)|}{|\B_{G,X}(n)|}=\frac{1}{[G:H]}.
	\]

	Furthermore, if \(H\) is an infinite index subgroup of \(G\), then both limits are zero for any coset of \(H\). 
\end{prop}

\begin{rem}
	The last statement does not appear explicitly in \cite{indexformula}, but follows easily from their arguments. Alternatively, one could prove this via the construction of an invariant mean which requires the choice of an ultrafilter. The stable subexponential condition ensures that any ultrafilter will do, and hence that all limit points of the sequences above are equal.
\end{rem}

\begin{lem}
	Let \(G\) be a group of stable subexponential growth with generating set \(X\), let \(g \in G\) and let \(H\) be a finite index subgroup of \(G\). For \(d \in \N\) we have
	\[
		\lim_{n\rightarrow\infty}\frac{|gH\cap\B_{G,X}(n+d)|}{|\B_{G,X}(n)|}=\frac{1}{[G:H]}.
	\]
\end{lem}

From now on, whenever there is no ambiguity concerning the group and its generating set, we will write \(\B(n)\) instead of \(\B_{G, X}(n)\).

\subsection{Non-virtually abelian groups of stable subexponential growth}	\label{stsubexp}
Antol\'{i}n, Martino and Ventura have shown that the commutativity degree \cite[Proposition~2.3]{dcA} of a residually finite group of subexponential growth is non-zero if and only if the group is virtually abelian.

We generalise this result to the degree of twisted commutativity.
In this section, we prove that residually finite non-virtually abelian groups have twisted degree of commutativity equal to zero.
To do so, we mimic the proof of Antol\'in, Martino and Ventura, and adapt it where necessary to include the endomorphism \(\phi\).

\begin{prop}	\label{prop:tdcOnFiniteQuotientIsLarger}
	Let \(G\) be a finitely generated group of stable subexponential growth and let \(\phi \in \End(G)\). Suppose that \(N\) is a normal subgroup of finite index in \(G\) such that \(\phi(N) \le N\). Let \(\overline{\phi}\) denote the induced endomorphism on \(G / N\). 
	
	Then \(\tdc_{X}(\phi, G) \le \tdc(\overline{\phi}, G/N)\) for any generating set \(X\) of \(G\).
\end{prop}

\begin{proof}
	Set \(d \coloneq [G : N]\) and fix a generating set \(X\) of \(G\).
	By \cref{pepprop}, the equality
	\[
		\lim_{n \to \infty} \frac{|gN \cap \B(n)|}{|\B(n)|} = \frac{1}{d}
	\]
	holds for all \(g \in G\).
	As there are only finitely many cosets of \(N\) in \(G\), this limit is even uniform in \(g\): for every \(\eps > 0\), there exists an \(n_{0} \in \Z_{>0}\) such that for all \(n \ge n_{0}\) and for all \(g \in G\),
	\begin{equation}	\label{eq:rewritingOfLimit}
		\left(\frac{1}{d} - \eps\right) \size{\B(n)} \le \size{gN \cap \B(n)} \le \left(\frac{1}{d} + \eps\right) \size{\B(n)}.
	\end{equation}
	Indeed, let \(g_{1}, \ldots, g_{k}\) be a full set of representatives of the cosets of \(N\) in \(G\).
	Given \(\eps > 0\), let \(n_{0, i}\) be such that \eqref{eq:rewritingOfLimit} holds for \(g_{i}\) and \(n \ge n_{0, i}\).
	Taking \(n_{0} \coloneq \max_{i} n_{0, i}\), we obtain \eqref{eq:rewritingOfLimit} for all \(n \ge n_{0}\) and \(g \in G\).
	
	Next, suppose for the sake of contradiction that \(\tdc_{X}(\phi, G) > \tdc(\overline{\phi}, G / N)\).
	This implies that there is a \(\delta > 0\) such that
	\begin{equation}	\label{eq:reverseInequality}
		\frac{|\{(x, y) \in \B(n)^2 :  xy=y\phi(x)\}|}{|\B(n)|^2} > \tdc(\overline{\phi}, G / N) + \delta
	\end{equation}
	for infinitely many integers \(n\).
	Take \(\eps > 0\) such that \(\eps d (2 + \eps d) \le \delta\) and let \(n_{0}\) be such that \eqref{eq:rewritingOfLimit} holds for all \(n \ge n_{0}\) and \(g \in G\).
	For \(n \ge n_{0}\) such that \eqref{eq:reverseInequality} holds, we first find
	\begin{small}
	\begin{align*}
		\tdc(\overline{\phi}, G / N) + \delta	&<	\frac{|\{(x, y) \in \B(n)^2 :  xy=y\phi(x)\}|}{|\B(n)|^2}	\\
		&\le \frac{1}{\size{\B(n)}^{2}} \size{\{(xN, yN) \in (G / N)^{2} : xNyN = yN \overline{\phi}(xN)\}} \left(\frac{1}{d} + \eps\right)^{2} \size{\B(n)}^{2},
	\end{align*}
	\end{small}%
	since \(xy = y \phi(x)\) implies that \(xN yN = yN \overline{\phi}(xN)\), and as every coset in \(G / N\) has at most \((1 / d + \eps) \size{\B(n)}\) representatives in \(\B(n)\).
	
	We work this out to obtain
	\begin{small}
\begin{align*}
	\tdc(\overline{\phi}, G / N) + \delta	&<	\frac{1}{\size{\B(n)}^{2}} \size{\{(xN, yN) \in (G / N)^{2} : xNyN = yN \overline{\phi}(xN)\}} \left(\frac{1}{d} + \eps\right)^{2} \size{\B(n)}^{2}	\\
		&=	\frac{|\{(xN, yN) \in (G / N)^{2} :  xNyN = yN \overline{\phi}(xN)\}|}{d^2} (1 + \eps d)^{2}	\\
		& \le	\frac{|\{(xN, yN) \in (G / N)^{2} :  xNyN = yN \overline{\phi}(xN)\}|}{d^2} + 2\eps d + \eps^{2} d^{2}	\\
		&= \tdc(\overline{\phi}, G / N) + \eps d (2 + \eps d),
\end{align*}
\end{small}

From this, we obtain the contradiction \(\delta < \eps d(2 + \eps d)\).

We conclude that \(\tdc_{X}(\phi, G) \le \tdc(\overline{\phi}, G / N)\).
\end{proof}

Next, we need two more technical lemmata.
\begin{lem}[{\cite[Lemma~2.2]{HsuWise03}}	]\label{lem:FullyCharacteristicFiniteIndexSubgroup}
	Let \(G\) be a finitely generated group and \(H\) a proper finite index subgroup.
	Then there is a fully characteristic finite index subgroup contained in \(H\).
\end{lem}

\begin{lem}[{\cite[Lemma~2.1]{dcA}}]	\label{lem:DCAtMostProductOfDcs}
	Let \(G\) be a finite group and \(N\) a normal subgroup of \(G\).
	Then
	\[
		\dc(G) \le \dc(N) \cdot \dc(G / N).
	\]
\end{lem}

We can now prove the first part of the generalisation of Antol\'in's, Martino's and Ventura's result.
\begin{thm}	\label{theo:tdcIsZeroForFGResFiniteGroupSubexpGrowthNonVirtuallyAbelian}
Let \(G\) be a finitely generated residually finite group of stable subexponential growth. Assume that \(G\) is not virtually abelian. Let \(\phi \in \End(G)\) and let \(X\) be a finite generating set.
Then 
\[
	\tdc_{X}(\phi, G) = \inf\{\tdc(\overline{\phi}, G / N) \mid N \lhd_{f} G, \phi(N) \le N, \overline{\phi} \text{ induced endo on } G / N\} =0.
\]
\end{thm}
\begin{proof}
	Since \(G\) is not abelian (being not virtually abelian), there are \(g, h \in G\) such that \([g, h] \coloneq \inv{g}\inv{h}gh \ne 1\).
	Furthermore, \(G\) is residually finite, so there is a finite index subgroup \(K_{1}\) such that neither of \(g, h,\) or $ [g, h]\) lies in \(K_{1}\).
	In particular, \(G / K_{1}\) is not abelian.
	By \cref{lem:FullyCharacteristicFiniteIndexSubgroup}, we may assume that \(K_{1}\) is fully characteristic.
	
	As a subgroup of a residually finite group, \(K_{1}\) is itself residually finite.
	As \(K_{1}\) has finite index in \(G\), it is finitely generated and not virtually abelian either.
	Hence, we can find a fully characteristic subgroup \(K_{2}\) of finite index in \(K_{1}\) such that \(K_{1} / K_{2}\) is not abelian.
	
	This way, we construct a descending sequence of subgroups
	\[
		K_{0} \coloneq G \rhd K_{1} \rhd K_{2} \rhd \dots \rhd K_{n - 1} \rhd K_{n} \rhd \dots
	\]
	such that each \(K_{i}\) is fully characteristic of finite index in its predecessor and such that each successive quotient is non-abelian.
	As a consequence, each \(K_{i}\) is fully characteristic of finite index in \(G\). \cref{lem:DCAtMostProductOfDcs,theo:GustafsonsBound} then
	imply that
	\[
		\dc(G / K_{i}) \le \dc(K_{i - 1} / K_{i}) \cdot \dc(G / K_{i - 1}) \le \frac{5}{8} \dc(G / K_{i - 1})
	\]
	for each \(i \ge 1\).
	Inductively, \(\dc(G / K_{i}) \le (5 / 8)^{i}\) for each \(i \geq 1\).
	
	Combining this with \cref{prop:tdcOnFiniteQuotientIsLarger} and \cref{cor:tdcAtMostDc}, we get
	\[
		\tdc_{X}(\phi, G) \le \tdc(\overline{\phi}_{i}, G / K_{i}) \le \dc(G / K_{i}) \le \left(\frac{5}{8}\right)^{i}
	\]
	for every \(i \ge 1\), where \(\overline{\phi}_{i}\) is the induced endomorphism on \(G / K_{i}\).
	We conclude that \(\tdc_{X}(\phi, G) = 0\). Moreover, as \(G / K_{i}\) is a finite group for each \(i \ge 1\), the infimum from the statement is zero as well.
	
\end{proof}

\subsection{Virtually abelian groups} \label{sec:tdc_VirtuallyAbelian}
In this section, we continue to generalise Antol\'in's, Martino's and Ventura's results, by giving equivalent conditions for the  degree of twisted commutativity to be non-zero for finitely generated virtually abelian groups.

Given a \fg{} virtually abelian group \(G\), we can find a \tf{} abelian fully invariant subgroup of finite index in \(G\) as follows: 
\(G\) has a finitely generated abelian subgroup \(A\) of finite index.
As \(A\) is finitely generated, it is virtually \tf{}.
It therefore contains a torsion-free subgroup \(B\) of finite index.

Next, we use \(B\) to construct a \tf{} abelian fully invariant subgroup of \(G\). Note that \(B\) has finite index in \(G\). Hence, by \cref{lem:FullyCharacteristicFiniteIndexSubgroup}, there exists a fully invariant subgroup \(K\) contained in \(B\) that has finite index in \(G\). Since \(K\) is a subgroup of \(B\), it is \tf{} and abelian as well.

We recall the definition of big-O and big-theta notation.

\begin{defn}
	Let \(f, g \colon \N \to \R\) be two functions.
	\begin{enumerate}[(i)]
		\item We write \(f(n) = O(g(n))\) if there exist \(c, M \in \R_{> 0}\) such that
		\[
			|f(n)| \le c |g(n)|
		\]
		for all \(n \ge M\).
		\item We write \(f(n) = \Theta(g(n))\) if there exist \(c_{1}, c_{2}, M \in \R_{> 0}\) such that
		\[
			c_{1} g(n) \le f(n) \le c_{2} g(n)
		\]
		 for all \(n \ge M\).
	\end{enumerate}
\end{defn}

\begin{thm}	\label{thm:CharacterisationTdcNonZeroVirtuallyAbelian}
	Let \(G\) be a finitely generated virtually abelian group and let \(X\) be a finite generating set.
	Let \(A\) be a \tf{} abelian fully invariant subgroup of finite index in \(G\) and fix \(\phi \in \End(G)\).
	
	Then \(\tdc_{X}(\phi, G) > 0\) if and only if there is \(g \in G\) such that \(\restrb{\inn{g} \circ \phi}{A} = \Id_{A}\).
	
	In particular, if \(Y\) is another finite generating set of \(G\), then \(\tdc_{X}(\phi, G) = 0\) if and only if \(\tdc_{Y}(\phi, G) = 0\).
\end{thm}

\begin{proof}
	Let \(\{g_{1}, \ldots, g_{m}\}\) be a full set of representatives of the cosets of \(A\) in \(G\). 
	
	For the first implication, suppose that \(g \in G\) is such that \(\restrb{\inn{g} \circ \phi}{A} = \Id_{A}\). Let \(i \in \{1, \ldots, m\}\) be such that \(gA = g_{i}A\). To prove that \(\tdc_{X}(\phi, G) > 0\), it is sufficient to show that, for this fixed \(i\),
	\begin{equation}	\label{eq:sufficientInequalitytdcPositive}
		\limsup_{n \to \infty} \frac{\sum_{x \in g_{i}A \cap \B(n)} \size{\Fix(\inn{x} \circ \phi) \cap A \cap \B(n)}}{\size{\B(n)}^{2}} > 0.
	\end{equation}
	Indeed, since
	\[
		\frac{\sum_{x \in \B(n)} \size{\Fix(\inn{x} \circ \phi) \cap \B(n)}}{\size{\B(n)}^{2}} \ge \frac{\sum_{x \in g_{i}A \cap \B(n)} \size{\Fix(\inn{x} \circ \phi) \cap A \cap \B(n)}}{\size{\B(n)}^{2}}
	\]
	and the limsup of the left-hand side equals \(\tdc_{X}(\phi, G)\) by \eqref{eq:tdcInFixedPoints}, it is enough to prove \eqref{eq:sufficientInequalitytdcPositive}.

	For a fixed value of \(n\), consider the summation
	\[
		\sum_{x \in g_{i}A \cap \B(n)} \size{\Fix(\inn{x} \circ \phi) \cap A \cap \B(n)}.
	\]
	First, remark that \(\Fix(\inn{x} \circ \phi) \cap A = \Fix(\restrb{\inn{x} \circ \phi}{A})\) for all \(x \in G\). Hence, the summation becomes
	\[
		\sum_{x \in g_{i}A \cap \B(n)} \size{\Fix(\restrb{\inn{x} \circ \phi}{A}) \cap \B(n)}.
	\]
	Next, for each \(a \in A\), \(\inn{a}\) is the identity map on \(A\). Hence, for all \(x \in g_{i}A = gA\), we find that \(\inn{x} = \inn{g_{i}} = \inn{g}\) on \(A\). Consequently,
	\begin{equation*}
		\sum_{x \in g_{i}A \cap \B(n)} \size{\Fix(\restrb{\inn{x} \circ \phi}{A}) \cap \B(n)} = \sum_{x \in g_{i}A \cap \B(n)} \size{\Fix(\restrb{\inn{g} \circ \phi}{A}) \cap \B(n)},	
	\end{equation*}
	and since \(\restrb{\inn{g} \circ \phi}{A} = \Id_{A}\), this last summation reduces to
	\begin{align*}
		\sum_{x \in g_{i}A \cap \B(n)} \size{\Fix(\restrb{\inn{g} \circ \phi}{A}) \cap \B(n)}		&= \sum_{x \in g_{i}A \cap \B(n)} \size{A \cap \B(n)}	\\
		&= \size{g_{i}A \cap \B(n)}\size{A \cap \B(n)}.
	\end{align*}
	
	Dividing by \(\size{\B(n)}^{2}\) and taking limsup, we find that
	\begin{align*}
		&\limsup_{n \to \infty}\frac{\sum_{x \in g_{i}A \cap \B(n)} \size{\Fix(\inn{x} \circ \phi) \cap A \cap \B(n)}}{\size{\B(n)}^{2}}	\\
		&= \limsup_{n \to \infty} \frac{\size{g_{i}A \cap \B(n)}\size{A \cap \B(n)}}{\size{\B(n)}^{2}}	\\
		&= \frac{1}{[G : A]^{2}} > 0
	\end{align*}
	by \cref{pepprop}. Hence, \(\tdc_{X}(\phi, G) > 0\).
	
	Conversely, suppose that \(\restrb{\inn{g} \circ \phi}{A} \ne \Id_{A}\) for all \(g \in G\). We again start from
	\[
		\sum_{x \in \B(n)} \size{\Fix(\inn{x} \circ \phi) \cap \B(n)}
	\]
	and rewrite it as follows:
	\[
		\sum_{x \in \B(n)} \size{\Fix(\inn{x} \circ \phi) \cap \B(n)} = \sum_{i = 1}^{m} \sum_{j = 1}^{m} \sum_{x \in g_{i}A \cap \B(n)} \size{\Fix(\inn{x} \circ \phi) \cap g_{j}A \cap \B(n)}.
	\]
	Fix \(x \in G\) and \(j \in \{1, \ldots, m\}\). Since \(\Fix(\inn{x} \circ \phi)\) and \(g_{j}A\) are both cosets of subgroups of \(G\), their intersection is either empty or a coset of \(\Fix(\inn{x} \circ \phi) \cap A = \Fix(\restrb{\inn{x} \circ \phi}{A})\). For \(i, j \in \{1, \ldots, m\}\) and \(n \ge 1\), let \(A_{i, j, n}\) be the set of \(x \in g_{i}A \cap \B(n)\) such that \(\Fix(\inn{x} \circ \phi) \cap g_{j}A = y_{x, j} \Fix(\restrb{\inn{x} \circ \phi}{A})\) for some \(y_{x, j} \in G\). The summation then becomes
	\[
		\sum_{i = 1}^{m} \sum_{j = 1}^{m} \sum_{x \in g_{i}A \cap \B(n)} \size{\Fix(\inn{x} \circ \phi) \cap g_{j}A \cap \B(n)} = \sum_{i = 1}^{m} \sum_{j = 1}^{m} \sum_{x \in A_{i, j, n}} \size{y_{x, j} \Fix(\restrb{\inn{x} \circ \phi}{A}) \cap \B(n)}.
	\]
	Now, fix \(i, j \in \{1, \ldots, m\}\) and \(x \in A_{i, j, n}\). By \cref{lem:FixedPointInfiniteIndexTFAbelianGroup}, since torsion-free abelian groups have unique roots, \(\Fix(\restrb{\inn{x} \circ \phi}{A})\) has infinite index in \(A \cong \Z^{r}\) (for some \(r \ge 1\)), so it is free abelian of rank \(r_{x, i, j}< r\). In particular, \(r_{x, i, j} \le r - 1\). Hence,
	\[
		\size{y_{x, j} \Fix(\restrb{\inn{x} \circ \phi}{A}) \cap \B(n)} = O\left(n^{r - 1}\right).
	\]
	We also have that
	\[
		\size{A_{i, j, n}} \le \size{g_{i}A \cap \B(n)} = O(n^{r}).
	\]
	This implies that
	\begin{align*}
		\sum_{x \in A_{i, j, n}} \size{y_{x, j} \Fix(\restrb{\inn{x} \circ \phi}{A}) \cap \B(n)} = O\left(n^{2r - 1}\right),
	\end{align*}
	and thus
	\[
		\sum_{i = 1}^{m} \sum_{j = 1}^{m} \sum_{x \in A_{i, j, n}} \size{y_{x, j} \Fix(\restrb{\inn{x} \circ \phi}{A}) \cap \B(n)} = O\left(n^{2r - 1}\right),
	\]
	as it is a finite sum of functions that are all \(O\left(n^{2r - 1}\right)\).
	
	So, if we consider
	\begin{align*}
		\frac{\sum_{x \in \B(n)} \size{\Fix(\inn{x} \circ \phi) \cap \B(n)}}{\size{\B(n)}^{2}} 
	\end{align*}
	and use the fact that \(\size{\B(n)} = \Theta(n^{r})\), we find that
	\[
		\frac{\sum_{x \in \B(n)} \size{\Fix(\inn{x} \circ \phi) \cap \B(n)}}{\size{\B(n)}^{2}} = O\left(n^{-1}\right).
	\]
	Consequently, \(\tdc_{X}(\phi, G) = 0\).

	Finally, the existence of a \(g \in G\) such that \(\restrb{\inn{g} \circ \phi}{A} = \Id_{A}\) is independent of the generating set. 
	Therefore, the degree of twisted commutativity of \(G\) with respect to \(\phi\) being \(0\) is independent of the generating set as well.
\end{proof}

To end this section, we prove that the (first) equality in \cref{theo:tdcIsZeroForFGResFiniteGroupSubexpGrowthNonVirtuallyAbelian} holds for virtually abelian groups as well if \(\tdc_{X}(\phi, G) = 0\).
We first need some results on Reidemeister numbers and a lemma about matrices.

\begin{lem}[{\cite[Proposition~2.1]{Senden22}}]	\label{lem:ReidemeisterNumberFiniteAbelianGroup}
	Let \(G\) be a finite abelian group and \(\phi \in \End(G)\). Then \(R(\phi) = \size{\Fix(\phi)}\).
\end{lem}

\begin{lem}[{\cite[Proposition~2.5.14]{Tertooy19}}]	\label{lem:BoundsReidemeisterNumberWithSummation}
	Let \(G\) be a group and \(N\) a finite index normal subgroup. Suppose that \(\phi \in \End(G)\) is such that \(\phi(N) \le N\). Let \(\overline{\phi}\) and \(\restr{\phi}{N}\) denote the induced endomorphisms on \(G / N\) and \(N\), respectively. Then
	\[
		\frac{1}{[G : N]} \sum_{gN \in G / N} R(\restrb{\inn{g} \circ \phi}{N}) \le R(\phi) \le \sum_{gN \in G / N} R(\restrb{\inn{g} \circ \phi}{N}).
	\]
\end{lem}

\begin{lem}	\label{lem:MatrixEqualIModp}
	Let \(r \ge 1\) be an integer and let \(M \in \Z^{r \times r}\). If \(M \equiv I \bmod p\) for infinitely many primes \(p\), then \(M = I\).
\end{lem}
\begin{proof}
	The condition on \(M\) implies that the off-diagonal entries of \(M\) are congruent to \(0\) modulo infinitely many primes.
	Hence, they must be identically \(0\).
	Similarly, the diagonal entries of \(M\) are congruent to \(1\) modulo infinitely many primes, so they must be identically \(1\).
\end{proof}

The following is an analogue of \cref{theo:tdcIsZeroForFGResFiniteGroupSubexpGrowthNonVirtuallyAbelian} for virtually abelian groups.

\begin{thm}	\label{theo:InfimumTdcFiniteQuotientsVirtuallyAbelian}
	Let \(G\) be an infinite, \fg{} virtually abelian group and let \(A\) be a \tf{} abelian fully invariant subgroup of finite index in \(G\). Let \(\phi \in \End(G)\). Suppose that \(\restrb{\inn{g} \circ \phi}{A} \ne \Id_{A}\) for all \(g \in G\). Then
	\[
		\inf\{\tdc(\overline{\phi}, G / N) \mid N \lhd_{f} G, \phi(N) \le N, \overline{\phi} \text{ induced endomorphism on } G / N\} = 0.
	\]
In particular, this infimum is also equal to \(\tdc_{X}(\phi, G)\) for every generating set \(X\) of \(G\).
\end{thm}

\begin{proof}
	By \cref{prop:tdcFiniteGroupsReidemeisterNumber}, \(\tdc(\overline{\phi}, G / N) = R(\overline{\phi}) / [G : N]\) for each normal \(\phi\)-invariant subgroup of finite index of \(G\). We will argue that we can get this ratio arbitrarily small for a given \(\phi \in \End(G)\).
	
	Let \(\phi \in \End(G)\) be arbitrary and consider, for any prime number \(p\), the group \(A^{p} \coloneq \grpgen{a^{p} \mid a \in A}\). Then \(A^{p}\) is fully invariant in \(A\), which itself is fully invariant in \(G\). Consequently, \(A^{p}\) is fully invariant in \(G\). Let \(\overline{\phi}_{p}\) denote the induced endomorphism on \(G / A^{p}\) and \(\overline{\phi}_{p, A}\) the one on \(A / A^{p}\). Since \(A \cong \Z^{r}\) for some \(r \ge 1\), we can represent the endomorphism \(\restr{\phi}{A}\) by a matrix \(M \in \Z^{r \times r}\). The quotient group \(A / A^{p}\) is then isomorphic to \((\ZpZ)^{r}\), so the matrix representation of \(\overline{\phi}_{p, A}\) is given by \(M \bmod p\).
	
	We want to show that \(R(\overline{\phi}_{p}) / [G : A^{p}] \le \frac{1}{p}\) for infinitely many prime numbers \(p\). To that end, we claim that the set
	\[
		\P \coloneq \{p \text{ prime} \mid \forall g \in G: C_{g}M \not\equiv I \bmod p\}
	\]
	is infinite, where for each \(g \in G\), \(C_{g}\) is the matrix representation of the inner automorphism \(\inn{g}\) restricted to \(A\).
	
	Let \(\{g_{1}, \ldots, g_{m}\}\) be a complete set of coset representatives of \(A\) in \(G\). Since \(\restr{\inn{a}}{A} = \Id_{A}\) for each \(a \in A\), it follows that
	\[
		\restr{\inn{g_{i}a}}{A} = \restr{\inn{g_{i}}}{A}
	\]
	for all \(a \in A\) and \(i \in \{1, \ldots, m\}\). Hence, \(C_{g_{i}a} = C_{g_{i}}\) for all \(a \in A\) and \(i \in \{1, \ldots, m\}\), so  
	\[
		\P = \{p \text{ prime} \mid \forall i \in \{1, \ldots, m\}: C_{g_{i}}M \not\equiv I \bmod p\}.
	\]
	
	Let \(i \in \{1, \ldots, m\}\) be arbitrary. As \(\restrb{\inn{g_{i}} \circ \phi}{A}\) is not the identity on \(A\), \cref{lem:MatrixEqualIModp} implies that there are only finitely many primes \(p\) for which \(C_{g_{i}}M \equiv I \bmod p\). Hence, \(\P\) is a finite intersection of co-finite sets, so it is itself co-finite. In particular, \(\P\) is infinite.
	
	Next, fix \(p \in \P\). Then for each \(g \in G\), the map \(\inn{gA^{p}} \circ \overline{\phi}_{p, A}\) is not the identity on \(A / A^{p}\), as \(C_{g}M \not\equiv I \bmod p\). Thus \(\Fix(\inn{gA^{p}} \circ \overline{\phi}_{p, A})\) is a proper subgroup of \(A / A^{p}\), so it has size at most \(p^{r - 1}\). As \(R(\inn{gA^{p}} \circ \overline{\phi}_{p, A}) = \size{\Fix(\inn{gA^{p}} \circ \overline{\phi}_{p, A})}\) by \cref{lem:ReidemeisterNumberFiniteAbelianGroup}, we find that
	\[
		R(\inn{gA^{p}} \circ \overline{\phi}_{p, A}) \le p^{r - 1}
	\]
	for every \(g \in G\).
	
	By \cref{lem:BoundsReidemeisterNumberWithSummation} applied to \(G / A^{p}\) and \(A / A^{p}\),
	\[
		R(\overline{\phi}_{p}) \le \sum_{gA \in G / A} R(\inn{gA^{p}} \circ \overline{\phi}_{p, A}) \le [G : A] p^{r - 1}.
	\]
	Here we use \(\frac{G / A^{p}}{A / A^{p}} \cong G / A\) to simplify the notation in the summation.
	This way, we get
	\[
		\tdc(\overline{\phi}_{p}, G / A^{p}) = \frac{R(\overline{\phi}_{p})}{[G : A^{p}]} \le \frac{[G : A] p^{r - 1}}{[G : A][A : A^{p}]} = \frac{p^{r - 1}}{p^{r}} = \frac{1}{p}.
	\]
	
	Since \(\P\) is infinite, we can make this ratio arbitrary small. This proves that
	\[
		\inf\{\tdc(\overline{\phi}, G / N) \mid N \lhd_{f} G, \phi(N) \le N, \overline{\phi} \text{ induced endomorphism on } G / N\} = 0.
	\]
	
	Finally, \cref{prop:tdcOnFiniteQuotientIsLarger} implies that
	\begin{align*}
		\tdc_{X}(\phi, G) &\le \inf\{\tdc(\overline{\phi}, G / N) \mid N \lhd_{f} G, \phi(N) \le N, \overline{\phi} \text{ induced endomorphism on } G / N\}	\\
			&= 0.
	\end{align*}
	Since \(\tdc_{X}(\phi, G)\) is non-negative, equality follows.
	Note that this matches with \cref{thm:CharacterisationTdcNonZeroVirtuallyAbelian}, as \(\restrb{\inn{g} \circ \phi}{A} \ne \Id_{A}\) for all \(g \in G\).
\end{proof}

\section{The twisted conjugacy ratio in groups}
We begin this section by recalling the notion of conjugacy ratio from \cite{CCM19}.
Let \(G\) be a finitely generated group with generating set $X$. The \emph{conjugacy ratio} of \(G\) with respect to \(X\) is defined as
\[
	\CR_{X}(G) \coloneq \limsup\limits_{n\rightarrow \infty} \frac{|C_{G,X}(n)|}{|\mathbb{B}_{G,X}(n)|}.
\]
We generalise this to the notion of \emph{twisted conjugacy ratio}.
\begin{defn}
    Let \(G\) be a finitely generated group and let \(X\) be a generating set. For \(\phi \in \mathrm{End}(G)\), let \(C^{\phi}_{G,X}(n)\) denote the set of \(\phi\)-conjugacy classes of \(G\) which have a representative in \(\mathbb{B}_{G,X}(n)\). The \emph{\(\phi\)-twisted conjugacy ratio} of \(G\) with respect to \(X\) is defined as 
    \[
    	\tcr_{X}(\phi, G) \coloneq \limsup\limits_{n\rightarrow \infty} \frac{|C^{\phi}_{G,X}(n)|}{|\mathbb{B}_{G,X}(n)|}.
    \]
\end{defn}

\begin{rem}\label{rmk:finite groups tcr tdc}
    For a finite group $F$, \(\tcr_{X}(\phi, F)\) does not depend on its generating set \(X\) and, moreover, by \cref{prop:tdcFiniteGroupsReidemeisterNumber} we get
    \[
    	\tcr(\phi, F) = \frac{R(\phi)}{\size{F}} = \tdc(\phi, F).
    \]
   
\end{rem}

For infinite groups, we can restrict our investigations to endomorphisms with infinite Reidemeister number.

\begin{lem}
	Let \(G\) be an infinite, finitely generated group and let \(\phi \in \End(G)\).
	Suppose that \(R(\phi) < \infty\).
	Then \(\tcr_{X}(\phi, G) = 0\) for every generating set \(X\).
\end{lem}

\begin{proof}
	Since \(|C_{G, X}^{\phi}(n)| \le R(\phi) < \infty\) for all \(n \ge 1\) and
	 \(\B_{G, X}(n)\) grows unboundedly, it follows that
	\[
		\tcr_{X}(\phi, G) = \limsup\limits_{n\rightarrow \infty} \frac{|C^{\phi}_{G,X}(n)|}{|\mathbb{B}_{G,X}(n)|} \le \limsup\limits_{n\rightarrow \infty} \frac{R(\phi)}{|\mathbb{B}_{G,X}(n)|} = 0.
	\]
\end{proof}

\subsection{Groups of stable subexponential growth} \label{sec:tcr_VirtuallyAbelian}

We first consider virtually abelian groups, and generalise \cite[Proposition 3.4]{CCM19}.
\begin{thm}\label{prop:tcrNonZeroVirtuallyAbelianGroups}
    Let \(G\) be a finitely generated virtually abelian group and \(A\) an abelian normal subgroup such that \([G : A] = m\).
    Let \(\phi \in \End(G)\) be such that \(\phi(A) \le A\) and \(X\) be a generating set of \(G\).
    Suppose there exists \(g \in G\) such that \(\restrb{\inn{g} \circ \phi}{A} = \Id_{A}\).
    Then \(\tcr_{X}(\phi, G) > 0\).
    
    More precisely, if \(g_{1}, \ldots, g_{m}\) are coset representatives of \(G / A\) and there exists an integer \(k \ge 1\) such that 
    \begin{enumerate}[(i)]
    	\item \(\restrb{\inn{g_{i}} \circ \phi}{A} = \Id_{A}\) for \(i \in \{1, \ldots, k\}\), and
	\item \([g_{i}A]_{\bar{\phi}} \ne [g_{j}A]_{\bar{\phi}}\) for \(i, j \in \{1, \ldots, k\}\) with \(i \ne j\), where \(\bar{\phi}\colon G / A \to G / A: gA \mapsto \phi(g)A\),
    \end{enumerate}
    then \(\mathrm{tcr}_{X}(\phi, G) \ge \frac{k}{m^{2}}\). 
\end{thm}

\begin{proof}	
	Fix \(n \ge 1\) and let \(k\) be as in the statement.
	We start by noting that
	\begin{align}
		\size{C^{\phi}_{G, X}(n)}	&=	\size{\{[g]_{\phi} \mid g \in \B_{G, X}(n)\}}	\nonumber\\
							&=	\size{\bigcup_{i = 1}^{m} \{[g]_{\phi} \mid g \in g_{i}A \cap \B_{G, X}(n)\}}	\nonumber	\\
							&\geq \size{\bigcup_{i = 1}^{k} \{[g]_{\phi} \mid g \in g_{i}A \cap \B_{G, X}(n)\}}	\label{eq:SubdivisionIntoDisjointParts}
	\end{align}
	We first prove that the last union is disjoint.
	
	Let \(i, j \in \{1, \ldots, k\}\) with \(i \ne j\).
	Then \([g_{i}A]_{\bar{\phi}} \ne [g_{j}A]_{\bar{\phi}}\).
	Note that, for \(g, h \in G\), the equality \([g]_{\phi} = [h]_{\phi}\) implies the equality \([gA]_{\bar{\phi}} = [hA]_{\bar{\phi}}\).
	Hence, contraposition implies that \([g_{i}a]_{\phi} \ne [g_{j}b]_{\phi}\) for \(a, b \in A\).
	Thus, \(\{[g]_{\phi} \mid g \in g_{i}A \cap \B_{G, X}(n)\}\) and \(\{[g]_{\phi} \mid g \in g_{j}A \cap \B_{G, X}(n)\}\) are disjoint.
	
	Next, we provide a lower bound on \(\size{\{[g]_{\phi} \mid g \in g_{i}A \cap \B_{G, X}(n)\}}\) for \(i \in \{1, \ldots, k\}\).
	Let \(i \in \{1, \ldots, k\}\).
	Then \(\restrb{\inn{g_{i}} \circ \phi}{A} = \Id_{A}\).
	Let \(h, \tilde{h} \in g_{i}A \cap \B_{G, X}(n)\) be such that \([h]_{\phi} = [\tilde{h}]_{\phi}\).
	Then there exists a \(g \in G\) such that \(h = g \tilde{h} \inv{\phi(g)}\).
	Write \(\tilde{h} = g_{i}\tilde{a}\) and \(g = g_{j}a\) for some \(a, \tilde{a} \in A\) and \(j \in \{1, \ldots, m\}\).
	Then
	\begin{align*}
		h	&=	g \tilde{h} \inv{\phi(g)}	\\
			&=	g_{j} a g_{i} \tilde{a} \inv{\phi(a)} \inv{\phi(g_{j})}	.
	\end{align*}
	Recall that \(\restrb{\inn{g_{i}} \circ \phi}{A} = \Id_{A}\).
	This means that \(a = g_{i} \phi(a) \inv{g_{i}}\), and thus 
	\[
		ag_{i} = g_{i} \phi(a).
	\]
	Therefore,
	\begin{align*}
		h	&=	g \tilde{h} \inv{\phi(g)}	\\
			&=	g_{j} (a g_{i}) \tilde{a} \inv{\phi(a)} \inv{\phi(g_{j})}	\\
			&=	g_{j} g_{i} \phi(a) \tilde{a} \inv{\phi(a)} \inv{\phi(g_{j})}	\\
			&=	g_{j} g_{i} \tilde{a} \inv{\phi(g_{j})}	\\
			&=	 g_{j} \tilde{h} \inv{\phi(g_{j})},
	\end{align*}
	where we used that \(\phi(a) \tilde{a} \inv{\phi(a)} = \tilde{a}\) since \(A\) is abelian. To summarise, for \(h, \tilde{h} \in g_{i}A \cap \B_{G, X}(n)\), we have \([h]_{\phi} = [\tilde{h}]_{\phi}\) if and only if \(h = g_{j} \tilde{h} \inv{\phi(g_{j})}\) for some \(j \in \{1, \ldots, m\}\).
	In other words, each \(\phi\)-conjugacy class in
	\[
		\{[g]_{\phi} \mid g \in g_{i}A \cap \B_{G, X}(n)\}
	\]
	occurs at most \(m\) times.
	This implies that
	\[
		\size{\{[g]_{\phi} \mid g \in g_{i}A \cap \B_{G, X}(n)\}} \ge \frac{\size{g_{i}A \cap \B_{G, X}(n)}}{m}.
	\]
		
	Going back to \eqref{eq:SubdivisionIntoDisjointParts}, we know that this union is disjoint.
	Hence,
	\begin{align*}
		\size{C^{\phi}_{G, X}(n)}		&\geq \size{\bigcup_{i = 1}^{k} \{[g]_{\phi} \mid g \in g_{i}A \cap \B_{G, X}(n)\}}	\\
								&=	\sum_{i = 1}^{k} \size	{\{[g]_{\phi} \mid g \in g_{i}A \cap \B_{G, X}(n)\}}	\\
								&\geq	\sum_{i = 1}^{k} \frac{\size{g_{i}A \cap \B_{G, X}(n)}}{m}.
	\end{align*}
	We thus find that
	\begin{align*}
		\tcr_{X}(\phi, G)	&\ge \limsup_{n \to \infty} \sum_{i = 1}^{k} \frac{\size{g_{i}A \cap \B_{G, X}(n)}}{\size{\B_{G, X}(n)} \cdot m}	\\
					&= \frac{1}{m}\sum_{i = 1}^{k} \limsup_{n \to \infty} \frac{\size{g_{i}A \cap \B_{G, X}(n)}}{\size{\B_{G, X}(n)}}	\\
					&= \frac{1}{m} \sum_{i = 1}^{k} \frac{1}{m}	\\
					&= \frac{k}{m^{2}},
	\end{align*}
	where the second and third equalities follow by \cref{pepprop}, as the \(\limsup\) of \(\frac{\size{g_{i}A \cap \B_{G, X}(n)}}{\size{\B_{G, X}(n)}} \) is actually a limit.
\end{proof}
\begin{remark}
	An alternative proof that \(\tcr_{X}(\phi, G) > 0\) under the conditions of \cref{prop:tcrNonZeroVirtuallyAbelianGroups} can be given using results by Dekimpe and Lathouwers.
	They proved \cite[Theorem~B]{DekimpeLathouwers25} that
	\begin{equation}	\label{eq:asymptoticGrowthDekimpeLathouwers}
		\size{C_{G, X}^{\phi}(n)} = \Theta\left(n^{\rank A - \min_{gA \in G / A} \rank(\Id_{A} - \restrb{\inn{g} \circ \phi}{A})} \right).
	\end{equation}	
	Suppose \(g \in G\) is such that \(\restrb{\inn{g} \circ \phi}{A} = \Id_{A}\).
	Then \(\rank (\Id_{A} - \restrb{\inn{g} \circ \phi}{A}) = 0\).
	Consequently, \eqref{eq:asymptoticGrowthDekimpeLathouwers} implies that \(\size{C_{G, X}^{\phi}(n)} = \Theta(n^{\rank A})\).
	As \(G\) is virtually abelian, \(\size{\B_{G, X}(n)} = \Theta(n^{\rank A})\) as well, which this implies that \(\tcr_{X}(\phi, G) > 0\).
\end{remark}

Next, we prove an inequality linking the twisted conjugacy ratio and the degree of twisted commutativity in groups of stable subexponential growth. 
\begin{prop}\label{tcr SSE}
    Let \(G\) be a finitely generated group of stable subexponential growth and \(N\) a normal subgroup of finite index in \(G\). Let \(\phi \in \mathrm{End}(G)\) such that \(\phi(N) \le N\). Then for any generating set \(X\) of \(G\), we have
    \[ \mathrm{tcr}_{X}(\phi, G) \le \mathrm{tdc}(\overline{\phi}, G / N),
    \]
    where \(\overline{\phi}\) is the induced endomorphism on \(G/N\). 
\end{prop}
The proof is similar to that of \cite[Proposition 3.6]{CCM19}, but we include it here for completeness.

\begin{proof}
    Let \([G:N] = m\) and write \(G = g_{1}N \sqcup g_{2}N \sqcup \dots \sqcup g_{m}N\), where \(g_{1}, \dots, g_{m} \in G\). Let \(d \coloneq \mathrm{max}\{|g_{i}|_{X}, |\phi(g_{i})|_{X} \; \colon \; i = 1,\dots, m \}\), where \(|.|_{X}\) denotes the word length with respect to \(X\). After reordering if necessary, we may assume that the distinct \(\overline{\phi}$-conjugacy classes of \(G/N\) are represented by \(g_{1}, \dots, g_{k}\), for some \(1 \le k \le m\). In particular, \(R(\overline{\phi}) = k\).

    Now consider \(y \in G\) such that \(|y|_{X} \le n\). Then there exist \(1 \le i \le m, 1 \le j \le k\) such that 
    \[ \phi(g_{j})^{-1}yNg_{j} = g_{i}N.
    \]
    Hence, there exists an \(h \in N\) such that \(\phi(g_{j})^{-1}yg_{j} = g_{i}h\), and so \(|g_{i}h|_{X} \le n+2d\). Thus, every element of \(\mathbb{B}(n)\) is \(\phi$-conjugate to some element of \(\sqcup^{k}_{i=1}(g_{i}N \cap \B(n+2d))\). Hence 
    \[ \frac{|C^{\phi}(n)|}{|\mathbb{B}(n)|} \le \frac{\sum^{k}_{i=1} \size{g_{i}N \cap \B(n+2d)}}{|\B(n)|}
    \]
    which tends to \(k/m\) by \cite[Lemma 3.5]{CCM19}. Also note that 
    \[ \mathrm{tcr}_{X}(\overline{\phi}, G/N) = \mathrm{tdc}_{X}(\overline{\phi}, G/N) = \frac{R(\overline{\phi})}{|G/N|} = \frac{k}{m}.
    \]
    The result then follows. 
\end{proof}

\cref{thm:tcr0} is a consequence of \cref{theo:tdcIsZeroForFGResFiniteGroupSubexpGrowthNonVirtuallyAbelian} and \cref{tcr SSE}.

\begin{thm}\label{thm:tcr0}
    Let \(G\) be a finitely generated residually finite group of stable subexponential growth that is not virtually abelian. Let \(\phi \in \mathrm{End}(G)\). Then for any generating set \(X\),
    \[
    	\tcr_{X}(\phi, G) = \tdc_{X}(\phi, G) = 0.
    \]
\end{thm}

We mention here recent work of Evetts \cite{Eve23} and Vandeputte \cite{Van25}, from which \cref{thm:tcr0} can also be derived for the class of generalised Heisenberg groups.

\subsection{Groups of exponential growth}
We conclude this section with several examples of groups of exponential growth and twisted conjugacy ratio equal to zero. We first consider the \emph{torus-knot groups} given by the presentation
\[ G = \langle x,y \mid x^{2} = y^{m} \rangle,
\]
where $m = 2k+1$ is an odd positive integer (these are also known as `odd dihedral Artin groups'). We will consider twisted conjugacy classes with respect to the automorphism $\phi \in \mathrm{Aut}(G)$ which sends $x \mapsto x^{-1}$ and $y \mapsto y^{-1}$. This was studied in detail by \cite{Crowe25}. Minimal length representatives of twisted conjugacy classes are of the form 
\begin{equation}\label{eqn:min form}
    w = x^{a_{1}}y^{b_{1}}\dots x^{a_\tau}y^{b_\tau},
\end{equation}
where 
\begin{enumerate}
    \item[$(i)$] $a_{1} = 0$ if and only if $b_{\tau} = 0$,
    \item[$(ii)$] $a_{i} = \pm 1$ for all $2 \le i \le \tau$, and
    \item[$(iii)$] $-k \le b_{i} \le k$ for all $1 \le i \le \tau-1$, $b_{i} \neq 0$ for $i = 2, \dots, \tau-1$.
\end{enumerate}
For a word $w$ of the form in \eqref{eqn:min form}, we say $w'$ is a $\phi$-cyclic permutation of $w$ if we can write $w = w_{1}w_{2}$ for some prefix and suffix $w_{1}$ and $w_{2}$ of $w$ respectively, such that either $w' = u_{2}\phi(u_{1})$, or $w' = \phi(u_{2})u_{1}$. 

Moreover, two words $w_{1}$ and $w_{2}$ of the form in \eqref{eqn:min form} are in the same twisted conjugacy class if and only if there exists a sequence of $\phi$-cyclic permutations and geodesic rewriting from $w_{1}$ to $w_{2}$, where the only rewriting rule for geodesics is swapping pairs of powers $x$ and $x^{-1}$ in a word.

\begin{prop}	\label{prop:DA}
    Let $G = \langle x,y \mid x^{2} = y^{m} \rangle$ be a torus-knot group with generating set $X = \{x^{\pm 1},y^{\pm 1}\}$, and let $\phi \in \mathrm
    {Aut}(G)$ be the order two automorphism which maps $x \mapsto x^{-1}$ and $y \mapsto y^{-1}$. Then $\mathrm{tcr}_{X}(\phi, G) = 0$. 
\end{prop}

\begin{proof}
First consider the projection $\pi \colon G \rightarrow Q = C_{2} \ast C_{2k+1}$. Note that if $g \in G$ is centreless, we have
\[ g \sim_{\phi} xgx \sim_{\phi} gx^{4}.
\]
In particular, $|C^{\phi}_{G,X}(n)| \leq 2|C^{\phi_{\pi}}_{Q, X'}(n)|$, where $\phi_{\pi}$ fixes $x$ and maps $y \mapsto y^{-1}$, and $X' = \{x, y, y^{-1}\}$. We therefore want to count $\phi_{\pi}$-conjugacy classes with representatives of the form
\begin{equation}\label{eqn:new form}
    w = xy^{b_{1}}xy^{b_{2}}\dots xy^{b_{\tau}}
\end{equation}
where $-k \leq b_{i} \leq k$, $b_{i} \neq 0$ for $i = 2, \dots, \tau-1$ and $0 < |b_{1}-b_{\tau}| \leq k$. 

Let $g(n)$ be the number of words of the form \eqref{eqn:new form} of length $n$, and let $f(n)$ denote the number of classes of $\phi$-cyclic permutations of the form \eqref{eqn:new form} of length $n$. Note $f(n)$ is an upper bound on $|C^{\phi}_{Q,X'}(n)|$. For each word of length $n$ of the form \eqref{eqn:new form}, there are $2n$ $\phi$-cyclic permutations we can apply, whilst remaining in the same twisted conjugacy class. This would suggest that $f(n) = g(n)/2n$, however this does not account for the fact that some $\phi$-cyclic permutations of length less than $n$ will result in us returning to the original word. 

To account for this, let $g_{d}(n)$ be the number of words of the form \eqref{eqn:new form} of length $n$, which are $d-\phi$-periodic, that is, the number of words with minimum period equal to $d$ under $\phi$-cyclic permutations. Note this can only occur when $d$ is a divisor of $2n$, by the definition of words of the form \eqref{eqn:new form} and that $\phi$ maps generators to their inverses. This implies that

    \[ |C^{\phi}_{Q,X'}(n)| \leq f(n) = \sum_{d | 2n} \frac{g_{d}(n)}{d} .
    \]
    First note that $g_{n}(n) = 1$ when $n=1$, otherwise $g_{n}(n) = 0$ since the exponents of $y$ are switched after $d=n$ $\phi$-cyclic permutations. This leaves us with
    \begin{equation}\label{eqn:5.8}
        \forall n\ge 2,\quad \size{C^{\phi}_{Q,X'}(n)} = \frac{g_{2n}(n)}{2n} + \sum_{d\mid 2n, d \leq \frac{2n}{3}} \frac{g_{d}(n)}{d}. 
    \end{equation}
    Since $g_{d}(n)$ is counting a subset of a regular language, the growth rate is bounded above by $q(n)\alpha^{d}$, for some polynomial $q$ and constant $\alpha > 1$. This leaves us with 
    \[ |C^{\phi}_{Q,X'}(n)| \leq \frac{q(n)\alpha^{2n}}{2n} + \sum_{d | 2n, d \leq \frac{2n}{3}} \frac{q(n)\alpha^{d}}{d} \leq   \frac{q(n)\alpha^{2n}}{2n} + \frac{2n}{3}q(n)\alpha^{2n/3}.
    \]
    Since the second term is exponentially smaller than $g_{2n}(n)$, we can rewrite \cref{eqn:5.8} as
    \[ |C^{\phi}_{Q,X'}(n)| \leq \frac{g_{2n}(n)}{2n} + \varepsilon
    \]
    where $\varepsilon$ is exponentially small. Since $|\mathbb{S}_{Q,X'}(n)| \geq g_{2n}(n)$, then 
    \[ \frac{f(n)}{|\mathbb{S}_{Q,X'}(n)|} \leq \frac{1}{2n} + \epsilon
    \]
    which tends to zero as $n$ tends to infinity. In particular, the spherical twisted conjugacy ration is zero, and so $\mathrm{tcr}_{X'}(\phi, Q) = 0$, by the Stolz-Ces\`aro theorem (see Remark 1.2 of \cite{CCM19}). The result then follows. 
\end{proof}
Our final examples arise from finite cyclic extensions. Let $H$ be a finitely generated group with generating set $X$, and consider the finite cyclic extension
\begin{equation}\label{eqn:extension}
    G = H \rtimes_{\phi} \Z/m\Z  = \langle X \cup t \mid R(H), t^{m}=1, t^{-1}xt = \phi(x) \; (x \in X) \rangle,
\end{equation}
where $H$ has presentation $\langle X \mid R(H) \rangle$, and $\phi \in \mathrm{Aut}(H)$ is of finite order dividing \(m\). With this construction, conjugacy classes of $G$ are linked to twisted conjugacy classes of $H$ in the following way.

\begin{lem}[{\cite[Page 4]{BMMV}}]	\label{lem:extension}
	Let $G$ be a group extension of the form \eqref{eqn:extension}. Let $t^{a}u, t^{b}v \in G$, with $u,v \in H$ and $0 \le a,b \le m-1$. Then $t^{a}u \sim t^{b}v$ if and only if $a = b$ and $v \sim_{\phi^{a}} \phi^{k}(u)$ for some integer $k$ such that $0 \le k \le m-1$.     
\end{lem}

This allows us to prove Proposition \ref{prop:finite extension}, which follows from results in \cite{CCM19} and Theorem \ref{thm:tcr0} for groups of stable subexponential growth, so is mainly of interest for groups of exponential growth. We let $\widehat{X} = X \cup \{t\}$ denote the generating set for $G$. 

\begin{prop}\label{prop:finite extension}
    Let $G$ be a finite group extension of $H$ as in \eqref{eqn:extension}, where $X$ is invariant under $\phi$. 
    Then $\CR_{\widehat{X}}(G) = 0$ if and only if $\tcr_{X}(\phi^{a}, H) = 0$ for all $0 \le a \le m-1$. 
\end{prop}

\begin{proof}
    We provide details when $m$ is odd, and the proof for $m$ even is analogous. 

    First note that any element $g \in G$ of the form in \eqref{eqn:extension} can be written as $g = t^{a}u$, for some $u \in H$, and $0 \leq a \leq m-1$. By \cref{lem:extension}, when counting conjugacy classes in $G$, we need to count twisted conjugacy classes 
    \[ [u], [u]_{\phi}, [u]_{\phi^{2}}, \dots,  [u]_{\phi^{m-1}}
    \]
    for all $u \in H$, and so 
    \[ \size{C_{G,\widehat{X}}(n)} = \size{C_{H,X}(n)} + \sum^{m-1}_{a=1} \size{C^{\phi^{a}}_{H,X}(n-a)}.
    \]
    Also, 
    \( \size{\B_{G,\widehat{X}}(n)} = \sum^{m-1}_{a=0} \size{\B_{H,X}(n-a)} \leq m \cdot \size{\B_{H,X}(n)}.
    \)
    
    For the forward direction, note that for all $0 \le a \le m-1$, $\size{C^{\phi^{a}}_{H,X}(n)} \le \size{C_{G, \widehat{X}}(n+a)}$. Therefore, for all $0 \le a \le m-1$, we have
    \begin{align*}
        \tcr_{X}(\phi^{a}, H) &= \limsup\limits_{n\rightarrow \infty} \frac{\size{C^{\phi^{a}}_{H, X}(n)}}{\size{\B_{H,X}(n)}}\\
        &\le \limsup\limits_{n\rightarrow \infty} \frac{\size{C_{G,\widehat{X}}(n+a)}}{\size{\B_{G,\widehat{X}}(n)}}\cdot m \\
        &= \limsup\limits_{n\rightarrow \infty} \frac{\size{C_{G,\widehat{X}}(n+a)}}{\size{\B_{G,\widehat{X}}(n+a)}}\cdot \frac{m \cdot \size{\B_{G,\widehat{X}}(n+a)}}{\size{\B_{G,\widehat{X}}(n)}} = 0
    \end{align*}
    since $\CR_{\widehat{X}}(G) = 0$ and $\frac{m \cdot \size{\B_{G,\widehat{X}}(n+a)}}{\size{\B_{G,\widehat{X}}(n)}}$ is bounded above by a constant as $n \rightarrow \infty$ and for fixed $a$.
    
    For the reverse direction, we have
    \begin{align*}
	\frac{\size{C_{G, \widehat{X}}(n)}}{\size{\B_{G,\widehat{X}}(n)}}	&= \frac{\size{C_{H,X}(n)} + \sum^{m-1}_{a=1} \size{C^{\phi^{a}}_{H,X}(n-a)}}{\sum^{m-1}_{a=0} \size{\B_{H,X}(n-a)}}\\
        &=\frac{\size{C_{H,X}(n)}}{ \sum^{m-1}_{a=0} \size{\B_{H,X}(n-a)}} + \frac{\size{C^{\phi}_{H,X}(n-1)}}{ \sum^{m-1}_{a=0} \size{\B_{H,X}(n-a)}} + \dots \\
        &\dots +  \frac{\size{C^{\phi^{m-1}}_{H,X}(n-(m-1))}}{ \sum^{m-1}_{a=0} \size{\B_{H,X}(n-a)}}                \\
        &\leq \frac{\size{C_{H,X}(n)}}{\size{\B_{H,X}(n)}}+ \frac{\size{C^{\phi}_{H,X}(n-1)}}{\size{\B_{H,X}(n-1)}} + \dots + \frac{\size{C^{\phi^{m-1}}_{H,X}(n-(m-1))}}{\size{\B_{H,X}(n-(m-1))}}	\\
        &= \sum^{m-1}_{a=0} \frac{\size{C^{\phi^{a}}_{H,X}(n-a)}}{\size{\B_{H,X}(n-a)}}.
    \end{align*}

    Each of the terms in the summation tends to zero as $n \rightarrow \infty$, since $\tcr_{X}(\phi^{a}, H) = 0$ for all $0 \le a \le m-1$.
    This implies that
    \[
    	\CR_{\widehat{X}}(G) = \limsup\limits_{n \to \infty} \frac{\size{C_{G, \widehat{X}}(n)}}{\size{\B_{G,\widehat{X}}(n)}} = 0,
    \]
	which completes the proof.
\end{proof}
This construction allows us to find examples of groups with exponential growth such that the twisted conjugacy ratio is zero. The following generalises \cite[Theorem 4.1]{dcA}, which is based on \cite[Theorem 1.2]{AC17}.

\begin{thm}\label{thm:hyp group tcr 0}
    Let $G$ be a hyperbolic group and let $\phi \in \mathrm{Aut}(G)$ be of finite order. Then $\tcr_{X}(\phi, G) = 0$ for any generating set $X$ of $G$. 
\end{thm}

\begin{proof}
    We can construct a finite extension $G' = G \rtimes_{\phi} \Z/m\Z$ of $G$ as in \eqref{eqn:extension}, where $G \le G'$ is of finite index. Since hyperbolicity is a quasi-isometry invariant, $\CR_{\widehat{X}}(G') = 0$, and so $\tcr_{X}(\phi, G) = 0$ by \cref{prop:finite extension}.
\end{proof}

\begin{exmp}\label{exmp:RAAG tcr 0}
   Let $A_{\Gamma}$ be a right-angled Artin group which is not free abelian, with standard generating set $X = V(\Gamma)$. 
   
   There exist examples of finite extensions of $A_{\Gamma}$ as in \eqref{eqn:extension} (besides the free non-abelian groups covered in \cref{thm:hyp group tcr 0}), for some finite order automorphisms $\phi \in \mathrm{Aut}(A_{\Gamma})$, which are graph products of cyclic groups (see \cite{Crowe23} and \cite[Prop. 4.2.32]{Crowe24}). Graph products of cyclic groups have standard growth $\sim e^{n}$ \cite[Theorem 6]{Val19} when the graph product is not a direct product, and the conjugacy growth is $\sim \frac{e^{n}}{n}$ by \cite[Corollary 1.7]{CG25}. Therefore $\CR_{\widehat{X}}(G) = 0$, and so $\tcr_{X}(\phi, A_{\Gamma}) = 0$ by \cref{prop:finite extension}.
\end{exmp}

\section{Reflections and open questions}\label{sec:reflections}

The results in this paper are split, roughly, into two types: those about groups of subexponential growth and those about groups of exponential growth. For the latter ones, we conjecture that \cref{prop:free} in \cref{sec:exp} can be generalised to (at least) torsion-free hyperbolic groups and their automorphisms.

\begin{quR}
	Is there an upper bound for the relative growth rates of fixed subgroups of (torsion-free) hyperbolic groups that is strictly smaller than the group growth rate?
\end{quR}
Fixed subgroups of automorphisms of hyperbolic group are quasi-convex \cite{Neu92}, and results on the relative growth of quasi-convex subgroups of infinite index in hyperbolic groups are given by \cite{DFW19}. Furthermore, in torsion-free hyperbolic groups, there are only finitely many isomorphism classes of fixed subgroups, so an argument as in \cref{prop:free} should be applicable.
The cases where all hyperbolic groups and all endomorphisms are considered, rather than torsion-free ones and just automorphisms, might require different techniques.

More broadly, our results likely hold for larger classes of groups of exponential growth: that is, we conjecture that the degree of twisted commutativity, as well as the twisted conjugacy ratio, are equal to zero in (non-abelian) RAAGs, non-affine Coxeter groups, relatively hyperbolic groups, and further acylindrically hyperbolic groups. 
Another class of interest is abelian-by-cyclic groups (see \cite{Guo25}).
In general, the study of the degrees of twisted commutativity or conjugacy is made difficult by the fact that we have very few tools to tackle twisted conjugacy growth.





Furthermore, we pose three questions:
\begin{quR}
Is the limsup in the definition of the twisted degree of commutativity or twisted conjugacy ratio a limit in general?
\end{quR}
As is the case for the degree of commutativity and the conjugacy ratio, we do not know whether the twisted versions might be influenced by a change of generators, or how they connect to each other.
\begin{quR}
Does there exist a group \(G\) with finite generating sets \(X\) and \(Y\) such that \(\tdc_X(\phi,G)\neq \tdc_Y(\phi, G)\), or $\tcr_{X}(\phi, G) \neq \tcr_{Y}(\phi, G)$ for fixed \(\phi\)?  
\end{quR}
\begin{quR}
What are the connections between \(\tdc_X(\phi,G)\) and $\tcr_{X}(\phi, G)$ for an arbitrary group $G$ and finite generating set $X$, where $\phi \in \End(G)$?  
\end{quR}
Corollary \ref{cor:eqtdctcr} shows that the above are equal for virtually abelian groups and arbitrary generating sets, but an intrinsic reason for this equality is not clear. Examples of groups where these two quantities differ are yet to be found.


%
%
%
%
%
%
%

 \section*{Acknowledgements}
The authors thank Corentin Bodart and Fran\c cois Dahmani for helpful discussions.

\appendix

\section{Twisted commutativity and conjugacy ratio in virtually abelian groups (by Corentin Bodart)}

In this appendix, we extend the arguments used in Sections \ref{sec:tdc_VirtuallyAbelian} and \ref{sec:tcr_VirtuallyAbelian} to obtain Theorems \ref{thm:va_indep} and \ref{thm:va_tcr_indep} about $\tdc$ and $\tcr$, respectively.

\begin{thm}\label{thm:va_indep}
	Let $G$ be a finitely generated virtually abelian group, and $\varphi\in\End(G)$. Then $\tdc_X(\varphi,G)$ is an effectively computable rational number, and
	\[ \tdc_X(\varphi,G) = \inf\{\tdc(\bar\varphi,G/N)\mid N\lhd_fG, \; \varphi(N)\le N\},\]
	where $\bar\varphi$ is the induced homomorphism on the finite group $G/N$.
	
	In particular, $\tdc_X(\varphi,G)$ is independent of $X$.
\end{thm}

We use the following setup for both results in this section.

\begin{notation}\label{not:appendix}
We may assume $G$ is infinite and let $A\lhd_f G$ be an abelian subgroup of finite index such that $A\simeq\Z^r$ and $\varphi(A)\le A$. We have a short exact sequence
	\[ 1 \longrightarrow A \longrightarrow G \overset\pi\longrightarrow F \longrightarrow 1 \]
	with $F$ finite. Since $\varphi(A)\le A$, the restriction of $\varphi$ gives an endomorphism $\bar\varphi\in\End(F)$. For each $i\in F$, we fix $g_i\in G$ such that $\pi(g_i)=i$. 

	\begin{itemize}
		\item[(i)] We define $L\in \Mat_r(\Z)$ such that $\varphi(a)= La$ for $a\in A$.
		\item[(ii)] For $i\in F$, we define $c_i\in A$ such that $\varphi(g_i)=g_{\bar\varphi(i)}\cdot c_i$.
		\item[(iii)] For $i\in F$, we define $M_i\in \GL_r(\Z)$ such that $ag_i=g_i(M_ia)$ for $a\in A$.
		\item[(iv)] For $i,j\in F$, we define $d_{i,j}\in A$ such that $g_ig_j=g_{ij}\cdot d_{i,j}$.
		\item[(v)] For $i\in F$, we define $e_i\in A$ such that $g_i^{-1}=g_{i^{-1}}\cdot e_i$.
	\end{itemize}
\end{notation}

\begin{proof}[Proof of \cref{thm:va_indep}]
	With the Notation above in place, the $\varphi$-commutativity condition for $g=g_ia$ and $h=g_jb$ can be rewritten as 
	\begin{align*}
		gh=h\varphi(g)
		& \iff g_iag_jb=g_jb\varphi(g_i)\varphi(a) \\
		& \iff \begin{cases}
			ij=j\bar\varphi(i) \\
			d_{i,j}+M_ja+b=d_{j,\bar\varphi(i)}+M_{\bar\varphi(i)}b+c_i+La
		\end{cases}
	\end{align*}
	For $i,j$ fixed, this last equation is an affine equation with unknowns $(a,b)\in A^2\simeq\Z^{2r}$, hence the set of solutions is either $A^2$ if
	\begin{equation} \label{eq:TrivialCondition}
		d_{i,j}=d_{j,\bar\varphi(i)}+c_i,\quad M_j=L \quad\text{and}\quad M_{\bar\varphi(i)}=I,
	\end{equation}
	or a proper subspace (or $\emptyset$) otherwise. We claim that
	\[\tdc(\varphi,G)=\frac1{\size F^2} \size{\left\{(i,j)\in F^2\mid ij=j\bar\varphi(i) \text{ and condition (\ref{eq:TrivialCondition}) holds}\right\}}.\]
	To prove the claim, write $T=\left\{(i,j)\in F^2\mid ij=j\bar{\phi}(i)\text{ and condition (\ref{eq:TrivialCondition}) holds}\right\}$ and fix  a generating set $X$ for $G$. We have
	\begin{align*}
		\tdc_X(\varphi,G)
		& = \limsup_{n\to\infty} \frac1{\size{\B_{G,X}(n)}^2} \cdot \size{\left\{(g,h)\in\B_{G,X}(n)^2 \mid gh=h\varphi(g)\right\}} \\
		& \ge \limsup_{n\to\infty} \frac1{\size{\B_{G,X}(n)}^2} \sum_{(i,j)\in T} \size{g_iA\cap \B_{G,X}(n)}\cdot\size{g_jA \cap \B_{G,X}(n) } \\
		& = \frac{\size T}{\size F^2}
	\end{align*}
	using \cref{pepprop}.
	
	\bigskip
	
	\noindent On the other side, for $p$ prime large enough, condition (\ref{eq:TrivialCondition}) holds modulo $p$ if and only if $(i,j)\in T$ (\cref{lem:MatrixEqualIModp}). Therefore, taking $N=pA$ (which is finite-index and fully characteristic in $A$, hence $N\lhd_f G$ and $\varphi(N)\leq N$), we have
	\begin{align*}
		\tdc(\bar\varphi,G/N) &
		\le \frac1{\size{G/N}^2}\left(\sum_{(i,j)\in T}p^{2r} + \sum_{(i,j)\notin T}p^{2r-1}\right) \\
		& = \frac{\size T\cdot p^{2r}+O(p^{2r-1})}{\size{F}^2p^{2r}} = \frac{\size T}{\size{F}^2}+O\left(\frac1p\right).
	\end{align*}
	Using \cref{prop:tdcOnFiniteQuotientIsLarger} and letting $p\to\infty$ implies that $\tdc_X(\varphi,G)\le\frac{\size T}{\size F^2\!\!}\,$.
\end{proof}
\begin{rem}
	This recovers \cref{thm:CharacterisationTdcNonZeroVirtuallyAbelian}. Indeed, one can take $g_1=1$ as representative of the coset $A$, in which case $d_{1,j}=d_{j,1}=c_1=0$. Moreover, we always have $\bar\varphi(1)=1$ and $M_1=I$. All in all, condition (\ref{eq:TrivialCondition}) is satisfied for some $(i,j)$ (equivalently, for some $(1,j)$) if and only if there exists $j\in F$ such that $M_j=L$, i.e., $\iota_{g_j}^{-1}=\varphi|_A$.
\end{rem}

\bigskip

We prove the analogous result for the twisted conjugacy ratio:
\begin{thm}\label{thm:va_tcr_indep}
	Let $G$ be a finitely generated virtually abelian group, and $\varphi\in\End(G)$. Then $\tcr_X(\varphi,G)$ is an effectively computable rational number, and
	\[ \tcr_X(\varphi,G) = \inf\{\tcr(\bar\varphi,G/N)\mid N\lhd_fG, \; \varphi(N)\le N\}.\]
	In particular, $\tcr_X(\varphi,G)$ is independent of $X$.
\end{thm}
\begin{proof}
	We use the same Notation as for the previous proof.
	For elements $g=g_ia$ and $x=g_jb$, twisted conjugation leads to the following computation:
	\begin{align*}
		xg\varphi(x)^{-1}
		& = g_jb\cdot g_ia \cdot \varphi(g_jb)^{-1} \\
		& = g_jg_i\cdot (M_ib)\cdot a \cdot (g_{\bar\varphi(j)}\cdot c_j\cdot Lb)^{-1} \\
		& = g_jg_i\cdot (M_ib)\cdot a \cdot (-Lb)\cdot (-c_j)\cdot g_{\bar\varphi(j)^{-1}}\cdot e_{\bar\varphi(j)} \\
		& = g_jg_ig_{\bar\varphi(j)^{-1}}\cdot \bigl(M_{\bar\varphi(j)}^{-1}(M_ib+a-Lb-c_j)+e_{\bar\varphi(j)}\bigr) \\
		& = g_{ji\bar\varphi(j)^{-1}} \cdot \bigl( M_{\bar\varphi(j)}^{-1}(a+(M_i-L)b-c_j)+d_{j,i\bar\varphi(j)^{-1}}+d_{i,\bar\varphi(j)^{-1}}+e_{\bar\varphi(j)}\bigr).
	\end{align*}
	To summarize, we have
	\begin{equation}
		xg\varphi(x)^{-1} = g_{ji\bar\varphi(j)^{-1}} \cdot \bigl( M_{\bar\varphi(j)}^{-1}(a+(M_i-L)b)+v_{i,j}\bigr) \label{eq:TwistedConjugateSpelledOut}
	\end{equation}
	where $v_{i,j}\coloneqq -M_{\bar\varphi(j)}^{-1}c_j +d_{j,i\bar\varphi(j)^{-1}} +d_{i,\bar\varphi(j)^{-1}} +e_{\bar\varphi(j)}$. We observe that $g$ has finitely many $\varphi$-conjugates if and only if $M_i=L$. In this case, there at most
	\[ C_i := \size{ \left\{\bigl( ji\bar\varphi(j)^{-1}, M_{\bar\varphi(j)}, v_{i,j}\bigr) \;\middle|\; j\in F\right\}}\]
	different $\varphi$-conjugates (since the only dependence left in \cref*{eq:TwistedConjugateSpelledOut} is through $j$ and more precisely through these parameters).
	We denote $U=\{i\in F\mid M_i=L\}$ and claim that
	\[ \tcr_X(\varphi,G) = \frac1{\size{F}} \sum_{i\in U}\frac1{C_i}.\]

	\bigskip
	
	\noindent Fix  a generating set $X$ for $G$. We have
	\begin{align*}
		\tcr_X(\varphi,G)
		& \ge \limsup_{n\to\infty} \frac1{\size{\B_{G,X}(n)}}\sum_{g\in \B_{G,X}(n)}\frac1{\size{[g]_\varphi}} \\
		& \ge \limsup_{n\to\infty} \frac1{\size{\B_{G,X}(n)}} \sum_{i\in U} \frac{\size{g_iA\cap \B_{G,X}(n)}}{C_i} \\
		& = \frac1{\size{F}} \sum_{i\in U}\frac1{C_i}
	\end{align*}
	by \cref{pepprop}. \bigskip
	
	\noindent Recycling the argument of \cref{lem:MatrixEqualIModp}, the following hold for $p$ large enough:
	\begin{itemize}
		\item $M_i\not\equiv L\pmod p$ for all $i\notin U$,
		\item $M_{\bar\varphi(j)}\not\equiv M_{\bar\varphi(k)}\pmod p$ for all $j,k$ such that $M_{\bar\varphi(j)}\ne M_{\bar\varphi(k)}$, and
		\item $v_{i,j}\not\equiv v_{i,k}\pmod p$ for all $i,j,k$ such that $v_{i,j}\ne v_{i,k}$.
	\end{itemize}
	We consider $N=pA$ (which is finite-index and fully characteristic in $A$, hence $N\lhd_f G$ and $N$ is $\varphi$-stable) for $p$ large enough. Using \cref{tcr SSE}, we have
	\[ \tcr_X(\varphi,G) \le \tcr(\bar\varphi,G/N) = \frac1{\size{G/N}}\sum_{gN\in G/N}\frac1{\size{[gN]_{\bar\varphi}}}. \]
	We need to find lower bounds on the number of different $\bar\varphi$-conjugates for each element $gN=g_iaN$ when $x=g_jb$ varies in \cref{eq:TwistedConjugateSpelledOut}:

	\begin{itemize}
		\item If $g=g_ia$ with $i\notin U$, then $M_i-L\not\equiv 0\pmod p$, hence $gN$ has at least $p$ different $\bar\varphi$-conjugates by just varying $b$ in \cref{eq:TwistedConjugateSpelledOut}, i.e., $\size{[gN]_{\bar\varphi}}\ge p$.
		
		\item If $g=g_ia$ with $i\in U$, then $M_i-L=0$ hence the dependence in $b$ in \cref{eq:TwistedConjugateSpelledOut} disappears. In particular, the number of $\bar\varphi$-conjugates of $gN$ is
		\begin{equation}	\label{eq:CountOfTwistedConjugates}
		\size{[gN]_{\bar\varphi}} = C_i
		\end{equation}
		\emph{except} if there exist $j,k\in F$ such that
		\begin{equation}
			M_{\bar\varphi(j)}^{-1}a + v_{i,j} \equiv  M_{\bar\varphi(k)}^{-1}a + v_{i,k} \pmod p \label{eq:SameTwistedConjugate}
		\end{equation}
		while $ji\bar\varphi(j)^{-1}=ki\bar\varphi(k)^{-1}$ and $(M_{\bar\varphi(j)},v_{i,j}) \ne (M_{\bar\varphi(k)},v_{i,k})$. For $i,j,k$ fixed, the affine equation (\ref{eq:SameTwistedConjugate}) admits at most $p^{r-1}$ solutions $a\in(\Z/p\Z)^r$. In total, we have at most $p^{r-1}\size F^3$ exceptions which still satisfy the trivial bound $\size{[gN]_{\bar\varphi}}\ge 1$.
	\end{itemize}
	Putting everything together, we get
	\begin{align*}
		\tcr_X(\varphi,G)
		& \le \frac1{\size{G/N}} \sum_{g\in G/N} \frac1{\size{[gN]_{\bar\varphi}}} \\
		& \le \frac1{p^r\size F} \left(\sum_{i\in U}
		\frac{\size{\bigl(g_iA/N-\{\text{exceptions}\}\bigr)}}{C_i}
		+ \frac{\size{\{\text{exceptions}\}}}1
		+ \sum_{i\notin U}\frac{\size{g_iA/N}}{p} \right) \\
		& \le \frac1{p^r\size F} \left(\sum_{i\in U} \frac{p^r}{C_i}
		+ \frac{p^{r-1}\size F^3}1
		+ \sum_{i\notin U}\frac{p^r}{p} \right) \\
		& =  \frac1{\size{F}} \sum_{i\in U}\frac1{C_i} + O\left(\frac1p\right).
	\end{align*}
	Letting $p\to\infty$ gives the matching upper bound on $\tcr_X(\varphi,G)$.
\end{proof}
\begin{rem}
	The fact that $\tcr_X(\varphi,G)$ is an effectively computable rational number already follows from results of Evetts and Lathouwers \cite[Theorem A]{EL24}.
\end{rem}
Combining Theorems \ref{thm:va_indep} and \ref{thm:va_tcr_indep} with \cref{rmk:finite groups tcr tdc}, we get
\begin{cor}\label{cor:eqtdctcr}
	Let $G$ be a finitely generated virtually abelian group, and $\varphi\in\End(G)$. Then $\tdc_X(\varphi,G)=\tcr_X(\varphi,G)$ for any finite generating set.
\end{cor}

\bigskip

\printbibliography

\newpage

\Address

\end{document}

